\documentclass[12pt]{amsart}

\usepackage{amsmath,amsfonts,amssymb,color,hyperref}
\usepackage{tikz}
\usepackage{float}
\usepackage{enumerate}

\numberwithin{equation}{section}

\theoremstyle{plain}
\newtheorem{thm}{Theorem}[section]
\newtheorem{lem}[thm]{Lemma}
\newtheorem{Def}[thm]{Definition}

\newtheorem{prop}[thm]{Proposition}

\newtheorem{conj}[thm]{Conjecture}

\newtheorem*{claim*}{Claim}
\newtheorem*{thm*}{Theorem}

\def\R{\mathbb{R}}

\def\Z{\mathbb{Z}}
\def\N{\mathcal{N}}
\def\M{\mathcal{M}}
\def\V{\mathcal{V}}
\def\dH{\dim_{\mathcal{H}}}
\def\dF{\dim_{\mathcal{F}}}

\newcommand{\supp}{\operatorname{supp}}

\newcommand{\dist}{\operatorname{dist}}
\newcommand{\esssup}{\operatornamewithlimits{ess\,sup}}
\newcommand{\essinf}{\operatornamewithlimits{ess\,inf}}

\begin{document}

\title[Dimension of Diophantine approximation and applications]{Dimension of Diophantine approximation and some applications in harmonic analysis}
\address{Department of Mathematics \& International Center for Mathematics, Southern
University of Science and Technology, Shenzhen, 518055, PR China}
\author{Longhui Li}
\email{12231267@mail.sustech.edu.cn}
\author{Bochen Liu}
\email{Bochen.Liu1989@gmail.com}

\begin{abstract}
    In this paper we construct a new family of sets based on Diophantine approximation in the Euclidean space, and consider their applications in several problems in harmonic analysis. 
 
    Our first application is on the Hausdorff dimension of our sets. We show a recent result of Ren and Wang on the ABC sum-product problem is sharp. Higher dimensional cases and the relation to orthogonal projections are also discussed. Some conjectures are proposed.
 
    In addition to Hausdorff dimension, we also consider Fourier dimension. For every $0\leq t\leq s\leq 1$, we are able to construct a subset of $\mathbb{R}$ that has Hausdorff dimension $s$ and Fourier dimension $t$, together with a measure $\mu$ that captures both dimensions, i.e.,
    $$\mu(B(x,r))\lesssim_\epsilon r^{s-\epsilon}
     \ \text{and} \ |\hat{\mu}(\xi)|\lesssim_\epsilon |\xi|^{-t/2 +\epsilon}, \ \forall\,\epsilon>0.$$
    It is fundamental but the very first such result in the literature.

    Our last result is to provide a viewpoint of the sharpness of Fourier restriction over general measures from dimensions of sets and measures.
\end{abstract}

\maketitle

\section{Introduction}
\subsection{Hausdorff dimension of Diophantine approximation}

Denote $\|x\|:=\dist(x,\Z)$ for $x\in\R$ and consider the Diophantine approximation
\begin{equation}
	\label{Jarnik}
	\{x\in\R: \|qx\|\leq q^{1-\alpha}\ \text{for infinitely many integers }q\}
\end{equation}
When $\alpha=2$, it is well known that it contains all real numbers. When $\alpha>2$, it is a real analysis exercise that this set has Lebesgue measure zero. In 1931, Jarnik \cite{Jar31} proved that the Hausdorff dimension of \eqref{Jarnik} equals $\min\{2/\alpha, 1\}$. A proof in English was given by  Besicovitch \cite{Bes34} in 1934. There is a large body of literature on Diophantine approximation from different aspects. 

In 1975, Kaufman and Mattila \cite{KM75} pointed out that the method of Jarnik implies the following uniform version: suppose $\{q_i\}_{i=1}^\infty$ is a rapidly increasing integer sequence, then the set
	\begin{equation}\label{KM}
 \begin{aligned}
     \bigcap_i\{x\in\R^d: \|Hx\|\leq Hq_i^{-\alpha}\ \text{for some integer }i\leq H\leq q_i\}
 \\=\bigcap_i\bigcup_{i\leq H\leq q_i}\N_{q_i^{-\alpha}}\left(\frac{\Z^d}{H}\right)
 \end{aligned}
\end{equation}
has Hausdorff dimension $\min\{(d+1)/\alpha, d\}$. Here and throughout this paper, $\|x\|:=\dist(x,\Z^d)$ for $x\in\R^d$, and $\N_\delta(\cdot)$ denotes the $\delta$-neighborhood of a set.

In the literature there is another construction that also appears  quite often: let $\{q_i\}$ be a rapidly increasing sequence, then the set
\begin{equation}
	\label{simple-example}
	\bigcap_i\{x\in\R^d: \|q_ix\|\leq q_i^{1-\alpha}\}=\bigcap_i\N_{q_i^{-\alpha}}\left(\frac{\Z^d}{q_i}\right)
\end{equation}
has Hausdorff dimension $\min\{d/\alpha, d\}$. For a proof we refer to Example 4.7 in \cite{Fal14}. Notice that \eqref{simple-example} has smaller dimension than \eqref{KM} for the same $\alpha$, as the former one has more flexibility on $H$.

Nowadays sets of type \eqref{Jarnik} are called limsup sets and \eqref{KM}, \eqref{simple-example} are called liminf sets. Unlike the case of limsup, there seems not much discussion on liminf sets (see, e.g. \cite{HS24}\cite{HW24}).

For a long time, \eqref{KM} and \eqref{simple-example} have been seen related but treated separately. They are considered as extremal cases for different problems in geometric measure theory. For instance, \eqref{simple-example} is used to propose the Falconer distance conjecture \cite{Fal85}, and the role of \eqref{KM} in \cite{KM75} is to construct examples on orthogonal projections. It seems their relation was never seriously discussed. In this paper we find they are actually endpoints of a family of sets.
\begin{thm}\label{thm-our-example}
Suppose $\gamma,\beta_1,\dots,\beta_d\geq 0$ and $0<\gamma+\beta_j<1$ for all $j$. Then there exists an increasing $\{q_i\}$ in $\R_+$ such that the set
	$$\begin{aligned}
	    \bigcap_i\{x\in\R^d: \|Hq_i^{\beta_j}x_j\|\leq Hq_i^{\beta_j-1}\, \text{for some integer }1\leq H\leq q_i^\gamma,  \,\forall\,j\}\\=\bigcap_i\bigcup_{1\leq H\leq q_i^\gamma}\N_{q_i^{-1}}\left(\frac{\Z}{Hq_i^{\beta_1}}\times\cdots\times\frac{\Z}{Hq_i^{\beta_d}}\right)
	\end{aligned}$$
	has Hausdorff dimension 
 $$\min\{(d+1)\gamma+\sum_{j=1}^d\beta_j,d\}.$$
 In particular it is sufficient to take $\{q_i\}$ as an increasing sequence in $(1,\infty)$ satisfying 
         \begin{equation}\label{how-fast}q_i>\max\{ q_{i-1}^{10di}, q_{i-1}^{\frac{1}{\gamma+\beta_j}}, 1\leq j\leq d \}.\end{equation}     
\end{thm}

The assumption $0<\gamma+\beta_j<1$ is natural: when $\gamma=\beta_j=0$, the $j$th coordinate only contains points in $\Z$; when $\gamma+\beta_j\geq 1$, by taking $H=q_i^\gamma$ one can see that every point in $\R$ satisfies the condition for the $j$th coordinate. In either case, it becomes a problem in $\R^{d-1}$. 

One can easily check that the construction in Theorem \ref{thm-our-example} is equivalent to \eqref{KM} when all $\beta_j$ vanish; it is equivalent to \eqref{simple-example} when $\gamma=0$ and all $\beta_j$s are equal.  Here we present our sets in a slightly different way from the classics to make this interpolation look natural. In fact, this is how we find it out.

\subsection{Orthogonal projection and sum-product}

\subsubsection{In the plane} Our first application is on orthogonal projections. For simplicity, we only introduce the history of the planar version and refer to \cite{Mat15} for all classical results. As mentioned in the previous subsection, Kaufman and Mattila considered \eqref{KM} because of orthogonal projections. More precisely, for $x\in\R^2$ and $e\in S^1$, let $\pi_e(x)=x\cdot e$ denote its orthogonal projection. In \cite{KM75}, for all $s\in(0,2)$, $t\in(0,1)$, Kaufman and Mattila construct Borel sets $E\subset\R^2$, $\dH E=s$, and $\Omega\subset S^1$, $\dH \Omega =t$, such that
\begin{equation}\label{sharp-orthogonal-proj-plane}
	\dH\pi_e(E)\leq \min\{\frac{s+t}{2}, s, 1\}, \ \forall\, e\in\Omega.
\end{equation}
In their argument $E\subset\R^2$ is taken as \eqref{simple-example} and $\Omega\subset S^1$ is determined by $(1, A)$ with $A\subset\R$ taken as \eqref{KM}.

On the other direction, for arbitrary Borel sets $E\subset\R^2$ and $\Omega\subset S^1$, it has been known for a while that there must exist $e\in\Omega$ such that
$$\dH\pi_e(E)\geq \begin{cases}1, & \text{if }\dH E+\dH\Omega> 2 \text{ \cite{Fal82}} \\ \dH E& \text{if }\dH\Omega>\dH E \text{ \cite{Kau68}}\end{cases},$$
and they are optimal due to \eqref{sharp-orthogonal-proj-plane}. For the remaining case, although people believe \eqref{sharp-orthogonal-proj-plane} should also be sharp, it was open for nearly half a century. More precisely, when $\dH E+\dH\Omega\leq 2$ and $0<\dH\Omega\leq\dH E$, there should exist $e\in\Omega$ such that
$$\dH\pi_e(E)\geq\frac{\dH E+\dH \Omega}{2}.$$
Finally, with the help of recent fast development in geometry measure theory and harmonic analysis, this is confirmed by Orponen-Shmerkin \cite{OS23-ABC} for the Ahlfors-David regular case and then fully solved by Ren-Wang \cite{RW23} for the general case. For more details we refer to their papers and references therein. 

Among all recent breakthroughs, a key point is the following ABC sum-product problem raised by Orponen in \cite{Orponen-ABC}: suppose $A, B, C\subset\R$ are Borel sets, $\dH A<1$, and 
\begin{equation}\label{C>A-B}\dH C>\dH A-\dH B\geq 0,\end{equation}
 then there should exist $c\in C$ such that
$$\dH (A+cB)>\dH A.$$
Here the dimensional threshold $\dH C>\dH A-\dH B\geq 0$ is necessary, by taking $A,B,C$ as \eqref{simple-example}. One proof of this problem was later given by Orponen and Shmerkin \cite{OS23-ABC}. 

We observe that, by treating $A+cB=\pi_{(1,c)}(A\times B)$ and applying the result of Ren-Wang on orthogonal projections, a more precise estimate on the $ABC$ sum-product problem follows, that is, under condition \eqref{C>A-B}, there exists $c\in C$ such that
\begin{equation}\label{eq-sharp-ABC}
\begin{aligned}
    \dH (A+cB)\geq  \min\{\frac{\dH (A\times B)+\dH C}{2}, \dH (A\times B), 1\}\\ \geq  \min\{\frac{\dH A+\dH B+\dH C}{2}, \dH A+\dH B, 1\}.
    \end{aligned}
\end{equation}
Here the last line follows from the well known property $\dH(A\times B)\geq \dH A+\dH B$. See, e.g., Section 8 in \cite{Mat95}.

Then a natural question is, whether the last line in \eqref{eq-sharp-ABC} is sharp in general.  In other words whether one should expect a better dimensional exponent on $\pi_e(E)$ under the extra Cartesian product assumption. With Theorem \ref{thm-our-example} we have the following.
\begin{thm}\label{thm-sharp-ABC}
	For all $s_A,s_B,s_C\in(0,1)$, $s_C>s_A-s_B\geq 0$,  there exist Borel sets $A,B,C\subset\R$ with $\dH A=s_A, \dH B=s_B, \dH C=s_C$ such that
	$$\dH (A+cB)\leq \min\{\frac{s_A+s_B+s_C}{2}, s_A+s_B, 1\},\ \forall\,c\in C.$$ 
\end{thm}
This completes the study of the ABC sum-product problem.

\subsubsection{Higher dimensions}
Now we turn to dimension $3$ and higher, in which the orthogonal projection is denoted by $\pi_V:\R^d\rightarrow V\in G(d,n)$, where $G(d,n)$ denotes the Grassmannian of $n$-dimensional subspaces in $\R^d$. 

In higher dimensions people used to construct examples from ``embedding". For instance, in $\R^3$ one can take $E\times\{0\}$, or $E=E'\times[0,1]$, $\Omega\subset S^2$ with all lines contained in $\R^2\times\{0\}$. One can combine these two constructions to obtain examples in every dimension. We refer to \cite{KM75}\cite{Gan24} for detailed discussions. 

Although the dimensional exponents look nice, these embedded examples are essentially planar. In this paper we would like to rule these out. The most natural sets not contained in any subspace is the Cartesian product $E=A_1\times\cdots\times A_d\subset\R^d$.

For the case $n=d-1$ we have the following generalization of Theorem \ref{thm-sharp-ABC}.
\begin{thm}\label{thm-codimension-1}
	For all $t\in (0, d)$, $s_1,\dots, s_d\in(0,1)$ with $s_1=\min s_j$,
 $$t>\sum_{j=2}^d(s_j-s_1),$$
 there exist Borel sets $A_1,\dots, A_d\subset\R$ with $\dH A_i=s_i, 1\leq i\leq d$, $\mathcal{V}\subset G(d,d-1)$ with $\dH\mathcal{V}=t$ and $\mathcal{V}^\perp$ not contained in a great circle, such that for all $V\in \mathcal{V}$,
	$$\dH \pi_V(A_1\times\cdots\times A_d)\leq \min\{\frac{(d-1)\sum s_i+t}{d}, \sum s_i, d-1\}.$$
 When $t\leq\sum_{j=2}^d(s_j-s_1)$, it becomes
 $$\dH \pi_V(A_1\times\cdots\times A_d)\leq s_2+\cdots+s_d,\ \forall\, V\in \mathcal{V},$$
 which matches the trivial lower bound.
\end{thm}
Compared to the planar case, it seems reasonable to expect the following.
\begin{conj}
    Suppose $E\subset\R^d$, $\mathcal{V}\subset G(d,d-1)$ are Borel sets, $\dH\mathcal{V}>0$ and $\mathcal{V}^\perp$ is not contained in a great circle, then there exists $V\in\mathcal{V}$ such that
    $$\dH \pi_V(E)\geq \min\{\frac{(d-1)\dH E+\dH\mathcal{V}}{d}, \dH E, d-1\}.$$
\end{conj}

When the codimension is greater than $1$, namely $n<d-1$, things get more complicated. In this case, our construction on $\dH\pi_V(A_1\times\cdots\times A_d)$ is determined by the vector $(\dH A_1, \dots, \dH A_d)$, not $\sum\dH A_i$. Because of this, it is difficult to make conjectures on $\dH\pi_V(E)$ from our Cartesian product example $E=A_1\times\cdots\times A_d$ to general $E\subset\R^d$. Also the Cartesian product structure on the direction set makes some difference when $n<d-1$. For example it seems one should expect different dimensional exponents on $$\dH \pi_e(A_1\times A_2\times A_3),\  e\in \Omega\subset S^2$$ and $$\dH (A_1+b_1A_2+b_2A_3),\  b_1\in B_1, b_2\in B_2,$$
even if $\dH\Omega=\dH (B_1\times B_2)$. We make the list for $d=3, n=1$ in Section \ref{subsec-codimension>1} to give readers some feeling.

Although the above somehow suggests that fully understanding higher dimensional orthogonal projections is challenging, we can still try to make some guesses. The followings are suggested by an anonymous reviewer. The idea is to remove some algebraic subset from $G(d,n)$ that needs to be figured out in the future.

\begin{conj}
    Let $E\subset\R^d$ be Borel, and let $\V\subset G(d,n)$ be Borel with $\dH\V>0$ and $V^\perp$ not contained in a proper algebraic subset. Then there exists $V\in\V$ such that
    $$\dH \pi_V(E)\geq\min\{\frac{n\dH E+\dH V}{d}, n\}$$
\end{conj}

\begin{conj}
If $E=A_1\times\cdots\times A_d\subset\R^d$ with $\dH A_j=s_j$, $s_1\leq\cdots\leq s_d$. Then for Borel $\V\subset G(d,n)$ with $\dH\V=t>0$ and $V^\perp$ not contained in a proper algebraic subset, there exists $V\in\V$ such that
    $$\dH \pi_V(E)\geq\min\{\sum_{j=d-n+1}^d s_j, \frac{n\sum_{j=1}^d s_j+t}{d}, n\}.$$
\end{conj}

\subsection{Fourier dimension of Diophantine approximation}
In addition to its Hausdorff dimension, the Diophantine approximation \eqref{Jarnik} is also famous for being a Salem set. To introduce the notion of Salem set we need to define the Fourier dimension. Here and throughout this paper, $\M(E)$ denotes the collection of nonzero finite Borel measures supported on a compact subset of $E$. Also $X\lesssim Y$ means $X\leq CY$ for some constant $C>0$, and $X\lesssim_\epsilon Y$ means this constant $C$ may depend on $\epsilon$.
\begin{Def}
    For a subset $E\subset\R^d$, its Fourier dimension is defined by
    $$\dF E:=\sup\{t\leq d:\exists\, \mu\in\M(E), s.t.\, |\hat{\mu}(\xi)|\lesssim |\xi|^{-t/2}\}.$$
\end{Def}
Due to an equivalent definition of the Hausdorff dimension through the energy integral (see, e.g. Section 2.5 in \cite{Mat15})
$$\dH E=\sup\{s: \exists\,\mu\in\M(E), s.t.\,\int|\hat{\mu}(\xi)|^2|\xi|^{-d+s}\,d\xi<\infty\},$$
one can conclude $\dim_{\mathcal{F}} E\leq \dH E$. A set is called Salem if the equality holds. In 1981, Kaufman \cite{Kau81} proved that the Diophantine approximation \eqref{Jarnik} is a Salem set, and so far it is still the only known explicit construction of Salem sets with arbitrary dimension in $\R$. In higher dimensions things are more complicated. Explicit Salem sets of arbitrary dimension in arbitrary $\R^d$ are not known until the recent work of Fraser-Hambrook \cite{FHR25}, and their construction relies on algebraic number theory. For more discussions we refer to their paper and references therein.

As one can imagine, Fourier dimension and Hausdorff dimension are not always equal. For example it is well known that the one-third Cantor set has Fourier dimension $0$. In fact the set \eqref{simple-example} at the beginning of this paper also has Fourier dimension $0$. To see this, as \eqref{simple-example} is constructed by neighborhoods of arithmetic progressions, $\dH(E+\cdots+E)=\dH E$ for every finite sum. On the other hand, if it has positive Fourier dimension, $\mu*\cdots*\mu$ would have fast Fourier decay and eventually ensure $E+\cdots +E$ to have positive Lebesgue measure. 

Then a natural question is, given arbitrary $0\leq t\leq s\leq 1$, does there exist a set in $\R$ of Hausdorff dimension $s$ and Fourier dimension $t$? This question alone is not very interesting, as one can just take a disjoint union of two compact sets, one is Salem of dimension $t$ and the other is \eqref{simple-example} of Hausdorff dimension $s$. However, this example is very hard to use, because people usually study sets with measures but there is no measure on this set that captures both dimensions, by which we mean
\begin{itemize}
     \item (Frostman condition) $$\mu(B(x,r))\lesssim_\epsilon r^{\dH E-\epsilon},\ \forall\,\epsilon>0;$$
     \item (Fourier decay) $$|\hat{\mu}(\xi)|\lesssim_\epsilon |\xi|^{-\frac{\dim_{\mathcal{F}}E}{2} +\epsilon}, \ \forall\,\epsilon>0.$$
 \end{itemize}
For the relation between Frostman condition and Hausdorff dimension, we again refer to Section 2.5 in \cite{Mat15}.

Based on the above, we rephrase the question to the following: given arbitrary $0\leq t\leq s\leq 1$, does there exist a set in $\R$ of Hausdorff dimension $s$ and Fourier dimension $t$, together with a measure that captures both dimensions? A related result was claimed by K\"orner in \cite{Kor11} using the Baire category method, but there is an error in the proof. We thank Nir Lev for pointing it out and starting a tripartite discussion with K\"orner. Eventually we all agree that the arguments in \cite{Kor11} work only for the case $s=t$. In this paper, we answer this question affirmatively for all $0\leq t\leq s\leq 1$.
\begin{thm}\label{thm-Fourier-dimension}
Suppose $\beta, \gamma\geq 0$ and $2\gamma+\beta\leq 1$. Then there exists an increasing $\{q_i\}$ in $\R_+$ such that the set
	$$E:=\begin{cases}
	    \bigcap_i\bigcup_{1\leq H\leq q_i^\gamma}\N_{q_i^{-1}}\left(\frac{\Z}{Hq_i^{\beta}}\right), & \text{if } 2\gamma+\beta<1\\
     \bigcap_i\bigcup_{1\leq H\leq q_i^\gamma, prime}\N_{q_i^{-1}}\left(\frac{\Z}{Hq_i^{\beta}}\right), & \text{if } 2\gamma+\beta=1
	\end{cases}$$
	has Hausdorff dimension $2\gamma+\beta$ and Fourier dimension $2\gamma$. Moreover, there exists a finite Borel measure $\mu$ supported on 
 $$\bigcap_i\bigcup_{q_i^\gamma/2\leq p\leq q_i^\gamma, \,  prime}\N_{q_i^{-1}}\left(\frac{\Z\backslash p\Z}{pq_i^{\beta}}\right)\cap[0,1]$$
 satisfying
$$\mu(B(x,r))\lesssim_\epsilon r^{\dH E-\epsilon}
     \ \text{and} \ |\hat{\mu}(\xi)|\lesssim_\epsilon |\xi|^{-\frac{\dim_{\mathcal{F}}E}{2} +\epsilon}, \ \forall\,\epsilon>0.$$
\end{thm}
The case $2\gamma+\beta=1$ is trickier as we need $E$ to have Lebesgue measure $0$. Then the following question sounds interesting:
\begin{center}
    {\it Does there exist $E\subset\R$ of positive Lebesgue measure and Fourier dimension $t$, for arbitrary $0\leq t\leq 1$, together with $f\in L^1(E)$ that achieves the optimal Fourier decay?}
\end{center}
The endpoint cases are already known, see Example 7 in \cite{EPS15} for $t=0$ and Theorem 2.1 in \cite{Shmerkin17} for $t=1$, while it is still unknown for $0<t<1$.

We believe that results similar to Theorem \ref{thm-Fourier-dimension} should hold in higher dimensions. But the construction may not be straightforward, as the only explicit Salem set we know is the algebraic construction of Fraser-Hambrook \cite{FH23}. 

\subsection{Sharpness of Fourier restriction estimates}

Our last application is on Fourier restriction estimates. Suppose $0<a, b <d$ and $\mu\in\M(\R^d)$ satisfying
$$\mu(B(x,r))\lesssim r^a \quad \text{and}\quad  |\hat{\mu}(\xi)|\lesssim |\xi|^{-b/2}.$$
The Mockenhaupt-Mitsis-Bak-Seeger Fourier restriction estimate states that
\begin{equation}\label{Fourier-restriction} \|\widehat{f\,d\mu}\|_{L^p(\R^d)}\lesssim_{p} \|f\|_{L^2(\mu)},\ \forall\,p\geq p_*(a,b, d):=\frac{4d-4a+2b}{b}.\end{equation}
This was independently proved by Mockenhaupt \cite{Moc00} and Mitsis \cite{Mit02} for $p>p_*(a, b, d)$, and the endpoint is due to Bak-Seeger \cite{BS11}. It is a generalization of the classical Stein-Tomas estimate, in which $\mu$ is the surface measure on $S^{d-1}$ with $a=b=d-1$. See also \cite{CS17}\cite{LW18}\cite{SS18}\cite{Ryou24}\cite{CFO24+} for related estimates.

Stein-Tomas is known be optimal, due to the famous Knapp's example, that is to consider small caps on the sphere. The sharpness of \eqref{Fourier-restriction} is, however, not this straightforward. 
Hambrook and Laba \cite{HL13} proved the optimality of $p_*(a,b,d)$ in \eqref{Fourier-restriction} when
$d = 1$ and $0 < b \leq a \leq d$. (In fact, Hambrook and Laba \cite{HL13} only claimed the case
$a=b$ for $a, b$ of a special form. The PhD thesis of Hambrook showed how optimality for all $b \leq a$ could be deduced from the result of \cite{HL13}. Later Chen \cite{Che16} modified the construction of \cite{HL13}
to obtain a technically stronger result that also implied optimality for $b \leq a$.) We refer to \cite{Laba14} for an expository paper, as well as \cite{HL16} for a higher dimensional result. Notice that all these constructions use randomness, and that none of them covered the case $b > a$. In this paper, we shall call $b\leq a$ the geometric case and $b>a$ the non-geometric case.

The first explicit sharpness examples for \eqref{Fourier-restriction} in the line, for the full range of $a,b$, were constructed recently due to Fraser-Hambrook-Ryou \cite{FHR25}. Later they generalized their result to higher dimensions \cite{FHR25+}.

In this paper, we investigate this sharpness from the perspective of dimensions. The geometric case $b\leq a$ and the non-geometric case $b>a$ are treated separately. 

\subsubsection{The geometric case}
For the geometric case, one can directly take sets and measures from Theorem \ref{thm-our-example}. We are able to obtain not only the sharpness of \eqref{Fourier-restriction}, but also a viewpoint from Hausdorff and Fourier dimensions of sets. This is why we call this case geometric.
\begin{thm}\label{thm-geo}
    Given $0<b\leq a< 1$, there exists a compact subset $E\subset [0,1]$ with $\dH E=a$, $\dF E=b$, and the following.
    \begin{enumerate}[(i)]
        \item There exists $\mu\in\M(E)$ satisfying
        $$\mu(B(x,r))\lesssim_\epsilon r^{a-\epsilon} \quad \text{and}\quad  |\hat{\mu}(\xi)|\lesssim_\epsilon |\xi|^{-b/2+\epsilon}.$$
        Consequently,
        $$\|\widehat{f\,d\mu}\|_{L^p(\R^d)}\lesssim_{p} \|f\|_{L^2(\mu)},\ \forall\,p>p_*(a,b,1)=\frac{4-4a+2b}{b}.$$
        \item For all $\mu\in\M(E)$ and all $p,q\in[1,\infty]$ satisfying
 $$p<\frac{2-2a+b}{b}q',\ \text{where }\frac{1}{q}+\frac{1}{q'}=1,$$
 the $L^q(\mu)\rightarrow L^p(\R)$ Fourier restriction fails, namely,
     $$\sup_{f\in L^q(\mu)}\frac{\|\widehat{f\,d\mu}\|_{L^p(\R)}}{\|f\|_{L^q(\mu)}}=\infty.$$
    \end{enumerate}
\end{thm}

\subsubsection{The non-geometric case}
Now we turn to the non-geometric case $b>a$. It is pointed out by Mitsis in \cite{Mit02} that it suffices to consider $a<b\leq 2a$. More precisely, if $|\hat{\mu}(\xi)|\lesssim |\xi|^{-b/2}$, then
$$\mu(B(x,r))\leq \int \phi(\frac{y-x}{r})\,d\mu(y)=r\int e^{2\pi i x\cdot\xi}\,\overline{\hat{\phi}(r\xi)}\,\hat{\mu}(\xi)\,d\xi$$ $$\leq r\int |\hat{\mu}(\xi)||\hat{\phi}(r\xi)|\,d\xi\lesssim r\int |\xi|^{-b/2}|\hat{\phi}(r\xi)|\,d\xi\lesssim r^{b/2}$$
for arbitrary $\phi\in C_0^\infty$ positive on the unit ball.

In the non-geometric case, there is no direct analog of Theorem \ref{thm-geo} because of $\dF E\leq \dH E$. This is also why we call it non-geometric. Instead of dimensions of sets, in this case we investigate the sharpness of Fourier restriction from dimensions of measures. In fact, this viewpoint is valid for all possible values of $a, b$.

Let
 \begin{equation}
     \label{def-H-dim-measure}
\dH\mu:=\inf\limits_{x\in\supp\mu}\left(\liminf\limits_{r\rightarrow 0}\frac{\log\mu(B(x,r))}{\log r}\right)\footnote{\text The quantity $\liminf\limits_{r\rightarrow 0}\frac{\log\mu(B(x,r))}{\log r}$ is usually called the lower local dimension of $\mu$ at $x$ and denoted by $\underline{\dim}(\mu,x)$, $\underline{\dim}_{loc}\mu(x)$, or $\underline{\dim}_{loc}(\mu,x)$.},
 \end{equation}
 \begin{equation}
     \label{def-F-dim-measure}
     \dF\mu:=\sup\{t: \sup_{|\xi|>1}|\hat{\mu}(\xi)||\xi|^{-t/2}<\infty\}.
 \end{equation}

\begin{thm}\label{thm-non-geo}
    Suppose $a,b\in(0,1)$, $b\leq 2a$. Then there exists a Borel measure on $[0,1]$ with $\dH \mu=a$, $\dF \mu=b$, and the following.
    \begin{enumerate}[(i)]
        \item The $L^2(\mu)\rightarrow L^p(\R)$ Fourier restriction holds for $p>p_*$, namely,
        $$\|\widehat{f\,d\mu}\|_{L^p(\R^d)}\lesssim_{p} \|f\|_{L^2(\mu)},\ \forall\,p>p_*(a,b,1)=\frac{4-4a+2b}{b}.$$
        \item For all $p,q\in[1,\infty]$ satisfying
 $$p<\frac{2-2a+b}{b}q', \ \text{where }\frac{1}{q}+\frac{1}{q'}=1,$$
 the $L^q(\mu)\rightarrow L^p(\R)$ Fourier restriction fails, namely,
     $$\sup_{f\in L^q(\mu)}\frac{\|\widehat{f\,d\mu}\|_{L^p(\R)}}{\|f\|_{L^q(\mu)}}=\infty.$$
    \end{enumerate}
\end{thm}

\subsection{Dimension of measures}\label{subsec-dim-measure}
In the previous subsection we discuss dimension of measures. For Fourier dimension, \eqref{def-F-dim-measure} seems to be the only reasonable definition. However, $\dH\mu$ in \eqref{def-H-dim-measure} looks slightly different from the commonly used Hausdorff dimension of measures in the literature:
\begin{equation}\label{fake-lower-Hausdorff-dime}\essinf_{x\sim \mu}\left(\liminf_{r\rightarrow 0}\frac{\log\mu(B(x,r))}{\log r}\right)=\inf\{\dH E: \mu(E)>0\}.\end{equation}
Some people use $\underline{\dim}_{\mathcal H}\mu$ to denote \eqref{fake-lower-Hausdorff-dime} as it is also called the lower Hausdorff dimension of $\mu$. For its basic properties and other related dimensions we refer to Chapter 10 in \cite{Falconer97}.

By comparing \eqref{def-H-dim-measure} and \eqref{fake-lower-Hausdorff-dime}, clearly $\dH\mu\leq \underline{\dim}_{\mathcal H}\mu$ because of $\inf\leq\essinf$. Then one may wonder if their difference really matters. The answer is yes in this paper. To study the non-geometric case we need $\dF\mu>\dH \mu$, but one can simply see that $\dF\mu\leq \underline{\dim}_{\mathcal H}\mu$:
$$|\hat{\mu}(\xi)|\lesssim |\xi|^{-t/2}$$ $$\implies \iint|x-y|^{-t'}d\mu(x)\,d\mu(y)=c\iint|\hat{\mu}(\xi)|^2|\xi|^{-d+t'}d\xi<\infty,\ \forall\,t'<t$$
$$\implies \int|x-y|^{-t'}d\mu(y):=C_x<\infty,\ \text{ for }\mu \text{-a.e. }x$$
$$\implies \mu(B(x,r))\leq C_xr^{t'},\ \forall\,r>0\text{ and }\mu \text{-a.e. }x$$
$$\implies \liminf_{r\rightarrow 0}\frac{\log\mu(B(x,r))}{\log r}\geq t', \ \text{ for }\mu \text{-a.e. }x.$$
Therefore the classical notion of Hausdorff dimension of measures \eqref{fake-lower-Hausdorff-dime} does not help us on the non-geometric case of the Fourier restriction.

\medskip
\subsection*{Organization}This paper is organized as follows. In Section \ref{sec-Hausdorff} we study the Hausdorff dimension and prove Theorem \ref{thm-our-example}. In Section \ref{sec-proj-sum-product} we consider applications on orthogonal projections and sum-product. In addition to the proof of Theorem \ref{thm-sharp-ABC}, \ref{thm-codimension-1}, we also discuss the case $d=3, n=1$ in detail to give readers some feelings on the complexity of orthogonal projections with codimension greater than $1$ (Proposition \ref{thm-d=3-n=1}, \ref{prop-prod-direction}). In Section \ref{sec-Fourier} we study the Fourier dimension. There will be three subsections: first we construct a measure with desired Fourier decay; then we show no measure could have faster Fourier decay; finally we show the measure constructed is also Frostman, thus complete the proof of Theorem \ref{thm-Fourier-dimension}. In Section \ref{sec-Sharpness} we prove Theorem \ref{thm-geo} and in Section \ref{sec-dim-measure} we prove Theorem \ref{thm-non-geo}.

\subsection*{Acknowledgement} Our original motivation is on the $ABC$ sum-product only. A discussion with De-Jun Feng encouraged us to consider the Fourier dimension. Then the second author attended a workshop in IBS Korea organized by Doowon Koh, Ben Lund and Sang-il Oum, where he learned the work of Fraser-Hambrook-Ryou on the Fourier restriction. We would like to thank their inspiration and encouragement.

We would like to thank Kevin Ren and an anonymous reviewer for discussion  and suggestion on orthogonal projections in higher dimensions. We also would like to thank Ben Ward, Baowei Wang, Nir Lev, Tom K\"orner, Izabella Laba, for discussion on earlier drafts.

\section{Hausdorff dimension of Diophantine approximation}\label{sec-Hausdorff}
    In this section we prove Theorem 1.1. The idea is not much different from the classic. But somehow we couldn't find a reference to help us skip some details. For example Jarnik's original paper \cite{Jar31} is not in English; there is no proof in the paper of Kaufman and Mattila \cite{KM75}; other classical sources like Besicovich's paper \cite{Bes34} and Falconer's book \cite{Fal14} only discuss $d=1$, while the higher dimensional case is a bit trickier (see below). Finally we decide to provide all the details, not only for the completeness, also hoping it can serve as a study guide for interested readers.
    
    For convenience we write the set as
    $$E=\bigcap_{i=1}^\infty E_i$$
    and denote
    \begin{equation}\label{def-s}s:=(d+1)\gamma+\sum_{j=1}^{d}\beta_{j}.\end{equation}
    It suffices to show
    $$\dim_{\mathcal{H}}E\cap[0,1)^d=s$$ given $s<d$. The case $s\geq d$ follows from the monotonicity of $E$ in $\gamma$.

    The upper bound is easy: every $E_i$ can be covered by no more than
    $$\sum_{1\leq H\leq q_i^\gamma} \prod_{1\leq j\leq d}H q_i^{\beta_j}\leq q_i^s $$
    cubes of side length $2q_i^{-1}$, thus $\dH E\leq s$. 
    
    For the lower bound, it is well known and easy to check that, if one can construct a Frostman measure on $E$, namely a finite Borel measure $\mu$ on $E$ satisfying
    $$\mu(B(x,r))\leq r^\alpha,\ \forall\,x\in\R^d,\,r>0,$$
    then the Hausdorff dimension of $E$ is at least $\alpha$. We refer to Theorem 2.7 in \cite{Mat15} for a reference.

\subsection{$\gamma=0$}\label{subsec-gamma=0}
We discuss the case $\gamma=0$ first as it is much simpler and already illustrate the main idea in the proof.

 When $\gamma=0$, the $j$th coordinate of each $E_i$ is just the $q_i^{-1}$-neighborhood of $q_i^{-\beta_j}\Z$ in $[0,1)$. As $q_i^{-\beta_j}<q_{i-1}^{-1}$ for all $j$ (recall \eqref{how-fast}, now $\gamma=0$), every $q_{i-1}^{-1}$-cube $Q$ in $E_{i-1}$ has a nonempty intersection with $E_i$. Moreover, by the lattice structure of $q_i^{-\beta_j}\Z$ in $E_i$, in $Q\cap E_i$ per coordinate the number of $q_i^{-1}$-intervals is 
 \begin{equation}\approx q_{i-1}q_i^{\beta_j}, j=1,\dots,d.\end{equation} Therefore the number of $q_i^{-1}$-cubes in $E_i\cap Q$ is
 \begin{equation}\label{counting-gamma=0}\approx \prod_{j=1}^d(q_{i-1}q_i^{\beta_j})=q_{i-1}^dq_i^{\sum_{j=1}^d\beta_j}=q_{i-1}^{-d} q_i^s,\end{equation}
where the implicit constants $0<c_d<C_d<\infty$ in \eqref{counting-gamma=0} are both independent in $i$. 
 
Then we construct our Frostman measure $\mu$ on $E$ as the following. Let $F_0=[0,1]^d$. Once $F_{i-1}$ is defined, inside every $q_{i-1}^{-1}$-cube in $F_{i-1}$ we pick exactly $c_d\,q_{i-1}^{-d} q_i^s$ many $q_i^{-1}$-cubes from $E_i$, and call the union of all chosen $q_i^{-1}$-cubes $F_i$. In particular,
    	\begin{equation}
    	    \label{total-gamma=0}
         \#\{q_{i}^{-1}\text{-cubes in }F_{i}\}=\prod_{k=1}^{i}c_d\,q_{k-1}^{-d}\,q_k^s.
    	\end{equation}
    Define
$$\mu_i:=\mathcal{H}^d(F_i)^{-1}\cdot\mathcal{H}^d|_{F_i},$$
    where $\mathcal{H}^d$ denote the $d$-dimensional Hausdorff measure. As each $q_{i-1}^{-1}$-cube $Q$ in $F_{i-1}$ contains the same amount of $q_i^{-1}$-cubes from $F_i$, one can conclude that for each $q_{i-1}^{-1}$-cube $Q$ in $F_i$,
    $$\mu_m(Q)=\left(\prod_{k=1}^{i}c_d\,q_{k-1}^{-d}\,q_k^s\right)^{-1},\ \forall\, m\geq i.$$ 
    In particular $\mu(Q)=\mu_i(Q)$ for every $q_i^{-1}$-cube in $F_i$, which guarantees that the weak limit of $\mu_i$ exists, denoted by $\mu$. Then $\mu$ is a probability measure supported on $\cap_{i=1}^\infty F_i\subset E$.
    
     We claim that $\mu$ is a Frostman measure of exponent $s'$ for any $s'<s$. This would complete our proof for $\gamma=0$.

Let $B(x,r)$ an arbitrary $r$-ball. Then there exists $i_0$ such that $$q_{i_0}^{-s/d}\leq r<q_{i_0-1}^{-s/d}.$$
    As $r\geq q_{i_0}^{-s/d}\geq q_{i_0}^{-1}$,
    \begin{equation}
    	\label{reduction-to-i-0-gamma=0}
    	\mu(B(x,r))\leq \mu(\bigcup_{\substack{Q\cap B(x,r)\neq\emptyset\\ q_{i_0}^{-1}\text{-cubes in } F_{i_0}}} Q)=\mu_{i_0}(\bigcup_{\substack{Q\cap B(x,r)\neq\emptyset\\ Q \text{ in } F_{i_0}}} Q)\leq \mu_{i_0}(B(x,C_d r)),
    \end{equation}
    and the problem is reduced to estimates on $F_{i_0}\cap B(x,r)$. There are two ways. First, by the lattice structure of $q_{i_0}^{-\beta_1}\Z\times\cdots\times q_{i_0}^{-\beta_d}\Z$ in    $E_{i_0}$ and $F_{i_0}\subset E_{i_0}$, in $F_{i_0}\cap B(x,r)$ per coordinate the number of $q_{i_0}^{-1}$-intervals is 
 \begin{equation}\lesssim r q_{i_0}^{\beta_j}, j=1,\dots,d.\end{equation}
 Therefore, with $s=\sum_{j=1}^d\beta_j$ defined in \eqref{def-s}, we have
    \begin{equation}\label{upper-r-gamma=0}\#\{q_{i_0}^{-1}\text{-cubes in }F_{i_0}\cap B(x,r)\}\lesssim \prod_{j=1}^d rq_{i_0}^{\beta_j}=C_d\,r^d q_{i_0}^s.\end{equation}    
    
    On the other hand, as $q_{i_0-1}^{-1}$-cubes in $F_{i_0-1}$ are $q_{i_0-1}^{-\min\beta_j}\geq q_{i_0-1}^{-s/d}$-separated, every $B(x,r)$ with $r<q_{i_0-1}^{-s/d}$ can intersect at most $C_d$ many $q_{i_0-1}^{-1}$-cubes from $F_{i_0-1}$, thus by \eqref{counting-gamma=0},
    \begin{equation}\label{upper-q-i-0-gamma=0}\#\{q_{i_0}^{-1}\text{-cubes in }F_{i_0}\cap B(x,r)\}\leq C_d\,q_{i_0-1}^{-d} q_{i_0}^s.\end{equation}
    By interpolating \eqref{upper-r-gamma=0} and \eqref{upper-q-i-0-gamma=0}, for every $s_{i_0}\in(0,d)$, we have
    \begin{equation}\label{upper-final-gamma=0}\#\{q_{i_0}^{-1}\text{-cubes in }F_{i_0}\cap B(x,r)\}\leq C_d\,r^{s_{i_0}} q_{i_0-1}^{-d+s_{i_0}} q_{i_0}^s.\end{equation}
    We need $s_i\rightarrow s$ from below, say 
    \begin{equation}\label{def-s-i}s_i:= s-\frac{1}{i}.\end{equation}
    
    Now we can estimate $\mu_{i_0}(B(x, r))$. By \eqref{total-gamma=0} and \eqref{upper-final-gamma=0},
    \begin{align*}
        \mu_{i_0}(B(x, r)) \leq & C_d\,r^{s_{i_0}}\,q_{i_0-1}^{s_{i_0}-d}\,q_{i_0}^s\left(\prod_{i=1}^{i_0}c_d\,q_{i-1}^{-d}q_i^s\right)^{-1}\\
        =&C_d \, r^{s_{i_0}}\, q_{i_0-1}^{s_{i_0}}\prod_{i=1}^{i_0-1}c_d^{-1} q_{i-1}^{d} q_i^{-s}\\=&C_d\, r^{s_{i_0}}\,q_{i_0-1}^{-\frac{1}{i_0}}\prod_{i=1}^{i_0-2}c_d^{-1}\cdot q_{i}^{d-s}.
    \end{align*}
        As $q_i>q_{i-1}^{10di}$ from our assumption \eqref{how-fast}, it follows that
        $$q_{i_0-1}^{-\frac{1}{i_0}}\prod_{i=1}^{i_0-2}c_d^{-1}\cdot q_{i}^{d-s}\leq C_d'$$
        and therefore
        $$\mu_{i_0}(B(x, r))\leq C_d C_d' \,r^{s-\frac{1}{i_0}},\ \text{given }q_{i_0}^{-s/d}\leq r<q_{i_0-1}^{-s/d}.$$
        Together with \eqref{reduction-to-i-0-gamma=0}, it follows that
        $$\mu(B(x, r))\leq C_d''\, r^{s-\frac{1}{i}},\ \forall\,
        q_{i}^{-s/d}\leq r<q_{i-1}^{-s/d}.$$ Consequently, for every $s'<s$, there exists a constant $C_{s',d}$ such that
        $$\mu(B(x,r))\leq C_{s',d}\,r^{s'},\ \forall\,x\in\R^d,\,r>0,$$
        as desired. 

\subsection{$\gamma>0$}
When $\gamma>0$, $E_i$ is no longer as well separated as $\gamma=0$, but we can still find a large well-separated subset. 

Let 
\begin{equation}\label{def-mathcal-P-i}\mathcal{P}_i^\gamma:=\{\text{primes in }[q_i^\gamma/2, q_i^\gamma]\}.\end{equation}  Consider
$$E_i':=\bigcup_{p\in\mathcal{P}_i^\gamma}\{x\in [0,1)^d: \|pq_i^{\beta_j}x_j\|\leq pq_i^{\beta_j-1}, \forall j\},$$
where $\|x\|:=\dist(x,\Z^d)$ in $\R^d$. For convenience we denote
    \begin{equation}\label{def-E-i-p}
E_{i,p}:=\{x\in [0,1)^d: \|pq_i^{\beta_j}x_j\|\leq pq_i^{\beta_j-1}, \forall j\}.
    \end{equation}
    In other words, $E_{i,p}$ is the union of $q_i^{-1}$-cubes centered at $\frac{\Z}{pq_i^{\beta_1}}\times \cdots\times \frac{\Z}{pq_i^{\beta_d}}$ in $[0, 1)^d$. Under these notation,
$$E_i'=\bigcup_{p\in\mathcal{P}_i^\gamma}E_{i,p}.$$    
    
    By the prime number theorem and \eqref{def-s}, the total number of $q_i^{-1}$-cubes in $E_i'$ with multiplicity is
    \begin{equation}\label{total-number-cubes}
    	\sum_{p\in\mathcal{P}_i^\gamma}\#\{q_i^{-1}\text{-cubes in }E_{i,p}\}\approx q_i^\gamma /\log q_i^\gamma\prod_{j=1}^d q_i^{\beta_j+\gamma} \approx q_i^s/\log q_i^\gamma.
    \end{equation}

    In the classical case $d=1$, $\beta=0$, one can directly see that these $q_i^{-1}$-cubes are well separated: for integers $m,m'\neq 0$,
    $$\left|\frac{m}{p}-\frac{m'}{p'}\right|=\frac{|mp'-m'p|}{pp'}\geq q_i^{-2\gamma}=q_i^{-s}, \ \forall\,(m,p)\neq(m',p').$$

    When $\beta>0$ or in higher dimensions this separation still holds on a large subset. When $d=1$ one can just drop at most $q_i^\beta$ many integers to consider $$E''_i:=\bigcup_{p\in\mathcal{P}_i^\gamma}\mathcal{N}_{q_i^{-1}}\left(\frac{\Z\backslash p\Z}{p q_i^\beta}\right)\cap[0,1),$$
    where $\N_\delta(\cdot)$ denotes the $\delta$-neighborhood. Then in $E_i''$, centers of cubes satisfy
    \begin{equation}
        \label{seperation-d=1}
        \left|\frac{m}{pq_i^\beta}-\frac{m'}{p'q_i^\beta}\right|=\frac{|mp'-m'p|}{pp'q_i^\beta}\geq q_i^{-2\gamma-\beta}=q_i^{-s}, \ \forall\,(m,p)\neq(m',p').
    \end{equation}
Notice $q_i^\beta\ll q_i^s/\log q_i^\gamma$ when $\gamma>0$, so $E_i'\backslash E_i''$ is a negligible subset of $E_i'$.

    For higher dimensions it requires more work to obtain such an $E''_i$. Let $\epsilon_i\rightarrow 0$ be a decreasing sequence in $(0,d-s)$, say, $$\epsilon_i:=\min\{\frac{1}{i}, d-s\}.$$
    
    For every $p\in\mathcal{P}_i^\gamma$ fixed, $q_i^{-1}$-cubes in $E_{i,p}$ are already $pq_i^{-\min \beta_j}\geq q_i^{-(s+\epsilon_i)/d}$ separated, for $i$ large enough in terms of $\gamma>0$, good enough. 
    
    For distinct primes $p\neq p'$ and integers $m,m'$, as $p'm=pm'$ forces $p|m$ and $p'|m'$, it follows that, for every integer $k$,    \begin{align*}
        \#\left\{(m_j, m_j')\in[0,pq_i^{\beta_j})\times[0,p'q_i^{\beta_j}):\;\frac{m_j}{p q_i^{\beta_j}}-\frac{m_j'}{p'q_i^{\beta_j}}=\frac{k}{pp'q_i^{\beta_j}}\right\}\leq q_i^{\beta_j}.
    \end{align*}
    Notice $\frac{m_j}{p q_i^{\beta_j}}-\frac{m_j'}{p'q_i^{\beta_j}}\in \frac{\Z}{pp'q_i^{\beta_j}}$ and $p, p'\leq q_i^\gamma$. So this implies that, for each $1\leq j\leq d$,
    \begin{align*}
        &\#\left\{(m_j, m_j')\in[0,pq_i^{\beta_j}]\times[0,p'q_i^{\beta_j}]:\;\left|\frac{m_j}{p q_i^{\beta_j}}-\frac{m_j'}{p'q_i^{\beta_j}}\right|\leq q_i^{-(s+\epsilon_i)/d}\right\}\\
        \leq & q_i^{\beta_j}\cdot \#\{k\in\Z:\left|\frac{k}{q_i^{\beta_j}pp'}\right|\leq q_i^{-(s+\epsilon_i)/d}\}\\
        \leq &4 pp'q_i^{2\beta_j-(s+\epsilon_i)/d}\leq 4 q_i^{2\gamma+2\beta_j-(s+\epsilon_i)/d}.
    \end{align*}
    Putting all $j$ together, we have that for all primes $p\neq p'$ in $(q_i^\gamma/2, q_i^\gamma)$,
    $$\#\left\{(\vec{m}, \vec{m}')\in \prod_{j=1}^d[0,pq_i^{\beta_j}]\times\prod_{j=1}^d[0,p'q_i^{\beta_j}]: \left|\frac{m_j}{q_i^{\beta_j}p}-\frac{m_j'}{q_i^{\beta_j}p'}\right|\leq q_i^{-(s+\epsilon_i)/d}, \;\forall j \right\}$$
    $$\leq \prod_{j=1}^d 4 q_i^{2\gamma+2\beta_j-(s+\epsilon_i)/d}=4^d q_i^{2d\gamma+2\sum_{j=1}^d\beta_j-s-\epsilon_i}= 4^d q_i^{s-2\gamma-\epsilon_i}.$$
    Fix $p\in\mathcal{P}_i^\gamma$ and let $p'\in \mathcal{P}_i^\gamma$ vary, it follows that for each $p\in\mathcal{P}_i^\gamma$,
    \begin{align*}
        \#\left\{\vec{m}\in \prod_{j=1}^d[0,pq_i^{\beta_j}]: \exists\,p'\in\mathcal{P}_i^\gamma, s.t., \dist\left(\frac{m_j}{q_i^{\beta_j}p}, \frac{\Z}{q_i^{\beta_j}p'}\right)\leq q_i^{-(s+\epsilon_i)/d}, \;\forall j \right\}
    \end{align*}
    $$\leq C_d\, q_i^{s-2\gamma-\epsilon_i}\cdot q_i^\gamma/\log q_i^\gamma=C_d\, q_i^{s-\gamma-\epsilon_i}/\log q_i^\gamma.$$
    By removing cubes centered at these $\vec{m}$ from $E_{i,p}$, we obtain a subset set $E_{i,p}'\subset E_{i,p}$, with
    \begin{equation}\label{large-subset-p}\#\{q_i^{-1}\text{-cubes in }E_{i,p}\backslash E_{i,p}'\}\leq C_d\, q_i^{s-\gamma-\epsilon_i}/\log q_i^\gamma\end{equation}
    and all $q_i^{-1}$-cubes in $\cup_{p\in\mathcal{P}_i^\gamma}E_{i,p}'$ are $q_i^{-(s+\epsilon_i)/d}$-separated.
    
    Finally take
    $$E_i'':=\bigcup_{p\in\mathcal{P}_i^\gamma}E_{i,p}'.$$
    By \eqref{large-subset-p} and the prime number theorem,
    \begin{equation}\label{large-subset}\begin{aligned}\#\{q_i^{-1}\text{-cubes in }E_i'\backslash E_i''\}\leq \sum_{p\in\mathcal{P}_i^\gamma}\#\{q_i^{-1}\text{-cubes in }E_{i,p}\backslash E_{i,p}'\}\\\leq C_d\, q_i^{s-\gamma-\epsilon_i}/\log q_i^\gamma\cdot q_i^{\gamma}/\log q_i^\gamma= C_d\, q_i^{s-\epsilon_i}/(\log q_i^\gamma)^2,\end{aligned}\end{equation}
    negligible to the total number of cubes in $E_i'$ (recall \eqref{total-number-cubes}). 
    
    As a summary, one can find a subset $E_i''\subset E_i'$ that consists of $c_d \,q_i^s/\log q_i^\gamma$ many $q_i^{-(s+\epsilon_i)/d}$-separated $q_i^{-1}$-cubes satisfying \eqref{large-subset}.
    
    We need more discussion on $E_i''$ before constructing a desired Frostman measure. Let $Q\subset [0, 1)^d$ be a $q_{i-1}^{-1}$-cube in $E_{i-1}''$ and consider the number of $q_i^{-1}$-cubes in $E_i''\cap Q$.

    The upper bound is again easy by the lattice structure of $E_{i,p}$ defined in \eqref{def-E-i-p}:
    \begin{equation}\label{counting-upper-gamma>0}\begin{aligned}\#\{q_i^{-1}\text{-cubes in }E_i''\cap Q\}&\leq \sum_p \#\{q_i^{-1}\text{-cubes in }E_{i,p}\cap Q\}\\&\leq C_d\, q_i^\gamma/\log q_i^\gamma\prod_{j=1}^d q_{i-1}^{-1}q_i^{\gamma+\beta_j}\\&= C_d\,q_{-1}^{-d} q_i^s/\log q_i^\gamma.\end{aligned}\end{equation}
    For the lower bound, by the separation on cubes in $E_{i}''$, the lattice structure of $E_{i,p}$ defined in \eqref{def-E-i-p}, and \eqref{large-subset-p}, we have
    \begin{equation}\label{first-term-second-term}\begin{aligned}
    	&\#\{q_i^{-1}\text{-cubes in }E_i''\cap Q\}\\=&\sum_p \#\{q_i^{-1}\text{-cubes in }E_{i,p}'\cap Q\}\\\geq& \sum_p(\#\{q_i^{-1}\text{-cubes in }E_{i,p}\cap Q\}-\#\{q_i^{-1}\text{-cubes in }E_{i,p}\backslash E_{i,p}'\})\\\geq & c_d (q_{i-1}^{-d}q_i^s/\log q_i^\gamma-q_i^{s-\epsilon_i}/(\log q_i^\gamma)^2).
    	\end{aligned}
    \end{equation}    
    Here the intersection $E_{i,p}\cap Q$ is nonempty because of $q_i>\max_j \{q_{i-1}^{\frac{1}{\gamma+\beta_j}}\}$ from our assumption \eqref{how-fast}.

    Recall $\epsilon_i=\min\{\frac{1}{i}, d-s\}$ and $q_i>q_{i-1}^{10di}$ from \eqref{how-fast}. So the second term in the last line of \eqref{first-term-second-term} is negligible  and therefore
    \begin{equation}\label{counting-lower-gamma>0}\#\{q_i^{-1}\text{-cubes in }E_i''\cap Q\} \geq c_d\,q_{i-1}^{-d}q_i^s/\log q_i^\gamma.\end{equation}
    
    Now one can construct our Frostman measure on $E$ in a similar way as $\gamma=0$. Let $F_0=[0,1]^d$. Once $F_{i-1}$ is defined, by \eqref{counting-lower-gamma>0} for every $q_{i-1}^{-1}$-cube $Q$ in $F_{i-1}$ one can pick exactly 
    \begin{equation}\label{exact-number-cubes-picked}c_d\,q_{i-1}^{-d} q_i^s/\log q_i^\gamma\end{equation} many $q_i^{-1}$-cubes in each $E_i''\cap Q$, and call the union of these $q_i^{-1}$-cubes $F_i$. In particular the total number of $q_i^{-1}$-cubes in each $F_i$ is exactly
    \begin{equation}
    	\label{total-i-0}
    	\prod_{k=1}^{i}c_d\,q_{k-1}^{-d}\,q_k^s/\log q_k^\gamma.
    \end{equation}

    Then we define
$$\mu_i=\mathcal{H}^d(F_i)^{-1}\cdot\mathcal{H}^d|_{F_i},$$
where $\mathcal{H}^d$ denote the $d$-dimensional Hausdorff measure. As each $q_{i-1}^{-1}$-cube $Q$ in $F_{i-1}$ contains the same amount of $q_i^{-1}$-cubes from $F_i'$, for each $Q$ in $F_i$,
    \begin{equation}\label{invariance-m-geq-i}\mu_m(Q)=\left(\prod_{k=1}^{i}c_d\,q_{k-1}^{-d}\,q_k^s/\log q_k^\gamma\right)^{-1},\ \forall\, m\geq i.\end{equation} 
    In particular $\mu(Q)=\mu_i(Q)$ for every $q_i^{-1}$-cube in $F_i$, which guarantees that the weak limit of $\mu_i$ exists, denoted by $\mu$. Then $\mu$ is a probability measure supported on $\cap_{i=1}^\infty F_i\subset E$.

The proof then goes like the case $\gamma=0$. Keep in mind that our $q_i^{-1}$-cubes are $q_i^{-(s+\epsilon_i)/d}$-separated, with $\epsilon_i=\min\{\frac{1}{i}, d-s\}$.

Let $B(x,r)$ an arbitrary $r$-ball. Then there exists $i_0$ such that $$q_{i_0}^{-\frac{s+\epsilon_{i_0}}{d}}\leq r<q_{i_0-1}^{-\frac{s+\epsilon_{i_0-1}}{d}}.$$
    As $\epsilon_i=\min\{\frac{1}{i}, d-s\}$, we have $q_{i_0}^{-1}\leq r$ and therefore by the discussion after \eqref{invariance-m-geq-i},
    \begin{equation}
    	\label{reduction-to-i-0}
    	\mu(B(x,r))\leq \mu(\bigcup_{\substack{Q\cap B(x,r)\neq\emptyset\\ q_{i_0}^{-1}\text{-cubes in } F_{i_0}}} Q)=\mu_{i_0}(\bigcup_{\substack{Q\cap B(x,r)\neq\emptyset\\ Q \text{ in } F_{i_0}}} Q)\leq \mu_{i_0}(B(x,C_d r)),
    \end{equation}
    and the problem is reduced to counting $F_{i_0}\cap B(x,r)$. There are two ways. First, by the lattice structure of $E_{i_0,p}$, we have
    \begin{equation}\label{upper-r}\begin{aligned}&\#\{q_{i_0}^{-1}\text{-cubes in }F_{i_0}\cap B(x,r)\}\\\leq &\sum_{p\in\mathcal{P}_i^\gamma}\#\{q_{i_0}^{-1}\text{-cubes in }E_{i_0, p}\cap B(x,r)\}\\\leq &C_d q_{i_0}^\gamma/\log q_{i_0}^{\gamma}\prod_{j=1}^d r q_{i_0}^{\gamma+\beta_j}= C_d\,r^d q_{i_0}^s/\log q_{i_0}^{\gamma}.\end{aligned}\end{equation}
    On the other hand, as cubes in $F_{i_0-1}$ are $q_{i_0-1}^{-\frac{s+\epsilon_{i_0-1}}{d}}$-separated, every $B(x,r)$ can intersect at most $C_d$ many $q_{i_0-1}^{-1}$-cubes from $F_{i_0-1}$, so by \eqref{counting-upper-gamma>0},
    \begin{equation}\label{upper-q-i-0}\#\{q_{i_0}^{-1}\text{-cubes in }F_{i_0}\cap B(x,r)\}\leq C_d\,q_{i_0-1}^{-d} q_{i_0}^s/\log q_{i_0}^{\gamma}.\end{equation}
    Then we interpolate \eqref{upper-r} and \eqref{upper-q-i-0} to conclude that, with $s_{i_0}:=s-\frac{1}{i_0}\in(0,d)$,
    \begin{equation}\label{upper-final}\#\{q_{i_0}^{-1}\text{-cubes in }F_{i_0}'\cap B(x,r)\}\leq C_d\,r^{s_{i_0}} q_{i_0-1}^{-d+s_{i_0}} q_{i_0}^{s}/\log q_{i_0}^{\gamma}.\end{equation}
    
    By \eqref{total-i-0}, \eqref{upper-final} and our definition of $\mu_i$,
    \begin{align*}
        &\mu_{i_0}(B(x, r))\\ \leq & C_d\,r^{s_{i_0}}\,q_{i_0-1}^{s_{i_0}-d}\,\frac{q_{i_0}^s}{\log q_{i_0}^{\gamma}}\left(\prod_{i=1}^{i_0}c_d\,q_{i-1}^{-d}q_i^s/\log q_i^{\gamma}\right)^{-1}\\
        =&C_d \, r^{s_{i_0}}\, q_{i_0-1}^{s_{i_0}}\prod_{i=1}^{i_0-1}c_d^{-1} q_{i-1}^{d}q_i^{-s} \log q_i^{\gamma}\\=&C_d\, r^{s_{i_0}}\,q_{i_0-1}^{-\frac{1}{i_0}}\log q_{i_0-1}^{\gamma}\prod_{i=1}^{i_0-2}c_d^{-1}\cdot q_{i}^s\log q_i^{\gamma},
    \end{align*}
    which is
    $$\leq C_d\, r^{s-\frac{1}{i_0}} $$
    because $q_i>q_{i-1}^{10di}$ from the assumption \eqref{how-fast}. Together with \eqref{reduction-to-i-0} we can conclude that
    $$\mu(B(x,r))\leq C_{d}\, r^{s-\frac{1}{i}},\ \forall\, q_{i}^{-\frac{s+\frac{1}{i}}{d}}\leq r<q_{i-1}^{-\frac{s+\frac{1}{i-1}}{d}}.$$
    This implies that for every $s'<s$, there exists a constant $C_{s,d}$ such that
        $$\mu(B(x,r))\leq C_{s,d}\,r^{s'},\ \forall\,x\in\R^d,\,r>0,$$
        as desired.

\section{Orthogonal projection and sum-product}\label{sec-proj-sum-product}
\subsection{In the plane}
We prove Theorem \ref{thm-sharp-ABC} first. The upper bound $1$ and $s_A+s_B$ are both trivial. To obtain $\frac{s_A+s_B+s_C}{2}$, we take $A, B$ to be as Theorem \ref{thm-our-example} with
$$\gamma_A=\gamma_B=0,\ \beta_A=s_A,\  \beta_B=s_B,$$
and $C\subset [0,1]$ as Theorem \ref{thm-our-example} with 
$$\gamma_C=\frac{s_C+s_B-s_A}{2},\ \beta_C=s_A-s_B,$$
under the same sequence $\{q_i\}$. Then for every $c\in C$ the set $A+cB$ is contained in the $q_i^{-1}$-neighborhood of 
$$\frac{\Z}{q_i^{s_A}}+\frac{n_0}{Hq_i^{s_A-s_B}}\cdot \frac{\Z}{q_i^{s_B}}\subset \frac{\Z}{Hq_i^{s_A}}$$
in $[0,2]$, for some integer $1\leq H\leq q_i^{\frac{s_C+s_B-s_A}{2}}$. Hence it can be covered by no more than $$2Hq_i^{s_A}\leq 2q_i^{\frac{s_A+s_B+s_C}{2}}$$ many intervals of length $q_i^{-1}$. 

Consequently, 
$$\dH(A+cB)\leq \frac{s_A+s_B+s_C}{2},\ \forall\,c\in C,$$
as desired.

\subsection{Higher dimensions: codimension $1$}
Now we prove Theorem \ref{thm-codimension-1}. Take $A_j, 1\leq j\leq d$, as in Theorem \ref{thm-our-example} with $\gamma=0, \beta=s_j$, and $E\subset\R^{d-1}$ as in Theorem \ref{thm-our-example} with
$$\gamma=\max\{\frac{t-\sum_{j=2}^d(s_j-s_1)}{d}, 0\}, \, \beta_j=s_{j+1}-s_1,\ j=1,\dots,d-1,$$
under the same sequence $q_i$. Then the set
$$\mathcal{V}:=\{V\in G(d,d-1):V^\perp\cap(\{1\}\times E)\neq\emptyset\}$$
has Hausdorff dimension $\dH\V=\dH E=t$.

Now the normal of every element $V\in \V$ lies in the $q_i^{-1}$-neighborhood of
$$(1, \frac{n_2}{Hq_i^{s_2-s_1}},\dots, \frac{n_d}{Hq_i^{s_d-s_1}}),\ \text{ for some }0\leq n_j\leq Hq_i^{s_j-s_1},$$
and all points of $A_1\times\cdots\times A_d$ lies in the $q_i^{-1}$-neighborhood of
$$(\frac{m_1}{q_i^{s_1}},\dots, \frac{m_d}{q_i^{s_d}}),\ 0\leq m_j\leq q_i^{s_j}.$$
Write $m_1=k H +h$, where $0\leq k\leq q_i^{s_1}/H$ and $0\leq h<H$, then
$$(\frac{m_1}{q_i^{s_1}},\dots, \frac{m_d}{q_i^{s_d}})=(\frac{kH+h}{q_i^{s_1}}, \frac{m_2}{q_i^{s_2}}, \dots, \frac{m_d}{q_i^{s_d}})$$
$$=\frac{kH}{q_i^{s_1}}\cdot(1,\frac{n_2}{Hq_i^{s_2-s_1}},\dots, \frac{n_d}{Hq_i^{s_d-s_1}})+(\frac{h}{q_i^{s_1}},\frac{m_2-kn_2}{q_i^{s_2}}, \dots, \frac{m_d-kn_d}{q_i^{s_d}}).$$
This implies the image of $\pi_V$ is determined by the second term in this sum. In other words  $\pi_V(A_1\times\cdots\times A_d)$ is contained in the $q_i^{-1}$-neighborhood of
$$\pi_V(\frac{[0,H]\cap\Z}{q_i^{s_1}}\times\frac{[-q_i^{s_2},q_i^{s_2}]\cap\Z}{q_i^{s_2}}\times\cdots\times\frac{[-q_i^{s_d},q_i^{s_d}]\cap\Z}{q_i^{s_d}}),$$
which, by trivial counting, can be covered by
$$\lesssim H q_i^{s_2+\cdots+s_d}\leq \begin{cases}
    q_i^{s_2+\cdots+s_d}, & \text{ if } t<\sum_{j=2}^d(s_j-s_1)\\q_i^{\frac{(d-1)(s_1+\cdots+s_d)+t}{d}}, & \text{ otherwise }
\end{cases} $$
many intervals of length $q_i^{-1}$.

Consequently,
$$\dH\pi_V(A_1\times\cdots\times A_d)\leq \begin{cases}
    s_2+\cdots+s_d, & \text{ if } t\leq \sum_{j=2}^d(s_j-s_1)\\\frac{(d-1)(s_1+\cdots+s_d)+t}{d}, & \text{ otherwise} 
\end{cases}
$$
for all $V\in\V$, as desired.
\subsection{Higher dimensions: codimension $>1$}\label{subsec-codimension>1}
As promised in the introduction we make the list for $d=3, n=1$.

\begin{prop}\label{thm-d=3-n=1}
    For all $t\in(0,2)$ and $1>s_1\geq s_2\geq s_3>0$ satisfying $t>2s_1-s_2-s_3\geq 0$, there exist Borel sets $A_1, A_2, A_3\subset\R$ and $\Omega\subset S^2$ not contained in a subspace, with $\dH A_i=s_i, 1\leq i\leq 3$, $\dH\Omega=t$, such that for all $e\in\Omega$,
	$$\dH \pi_e(A_1\times A_2\times A_3)\leq\min\{\frac{s_1+s_2+s_3+t}{3}, f(s_1,s_2,s_3,t), \sum s_i, 1\},$$
 where $f$ is a piecewise linear function
 $$f(s_1,s_2,s_3,t):=\begin{cases}
     s_1+s_3, & t\leq 1+s_1-s_2\\\frac{s_1+s_2+t-1}{2}+s_3, & t\geq 1+s_1-s_2
 \end{cases}.$$
\end{prop}
Here the role of $t>2s_1-s_2-s_3\geq 0$ is the same as \eqref{C>A-B}, otherwise there exists a construction with $\pi_e(A_1\times A_2\times A_3)=\dH A_1, \forall e\in\Omega$. We leave details to interested readers. As $\sum s_i$ and $1$ are trivial bounds, we only compare $\frac{s_1+s_2+s_3+t}{3}$ and $f(s_1, s_2, s_3, t)$. It turns out to be quite complicated. See the figure below. 
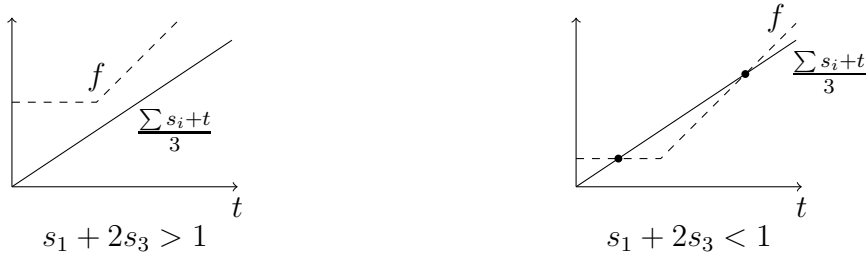
\begin{figure}[H]
\centering
\begin{tikzpicture}[scale=0.75, p2/.style={line width=0.275mm, black}, p3/.style={line width=0.15mm, black!50!white}]

\draw[->] (10, 0) -- (10, 3);
\draw[->] (10, 0) -- (14, 0);
\draw[-] (10, 0) -- (13.9, 2.6);
\draw[dashed] (10,0.5) -- (11.5, 0.5) -- (13.9, 2.9);

\fill (10.75,0.5) circle (0.07);
\fill (13,2) circle (0.07);

\draw (14, 0) node[anchor=north]{$t$};
\draw (12, -0.5) node[anchor=north]{$s_1+2s_3<1$};
\draw (13.9, 3) node[anchor=east]{$f$};
\draw (13.6, 2.6) node[anchor=north west]{$\frac{\sum s_i +t}{3}$};

\draw[->] (0, 0) -- (0, 3);
\draw[->] (0, 0) -- (4, 0);
\draw[-] (0, 0) -- (3.9, 2.6);
\draw[dashed] (0,1.5) -- (1.5, 1.5) -- (3,3);

\draw (2, -0.5) node[anchor=north]{$s_1+2s_3>1$};
\draw (4, 0) node[anchor=north]{$t$};
\draw (1.5, 1.5) node[anchor=south]{$f$};
\draw (2, 1) node[anchor=west]{$\frac{\sum s_i+t}{3}$};

\end{tikzpicture}
\caption{$\frac{\sum s_i+t}{3}$ and $f(s_1, s_2, s_3, t)$ as functions of $t$}
\label{fig:range-p}
\end{figure}

We also mention in the introduction about the Cartesian product structure on the direction set. Results for $d=3, n=1$ is given here for comparison with Proposition \ref{thm-d=3-n=1}. This is also where the piecewise linear function $f$ comes from.
\begin{prop}
    \label{prop-prod-direction}
    For all $t_1, t_2\in(0,1)$ and $1>s_1\geq s_2\geq s_3>0$, there exist Borel sets $A_1, A_2, A_3, B_1, B_2\subset\R$, with $\dH A_i=s_i, i=1,2,3$, $\dH B_j=t_j, j=1,2$, such that for all $b_j\in B_j, j=1,2$,
	$$\dH (A_1+b_1A_2+b_2A_3)$$ $$\leq\min\left\{s_1+\sum_{j=1}^2\min\left\{\max\{\frac{t_j+s_{j+1}-s_1}{2}, 0\}, s_j\right\}, 1\right\}.$$
 In particular, the upper bound is $s_1+s_3$ when $t_1\leq s_1-s_2$, $t_2=1$ and $\frac{s_1+s_2+t_1}{2}+s_3$ when $t_1\geq s_1-s_2$, $t_2=1$.
\end{prop}

We give the proof of Proposition \ref{prop-prod-direction} first, then prove Proposition \ref{thm-d=3-n=1}.

\begin{proof}[Proof of Proposition \ref{prop-prod-direction}]
The upper bounds $\sum s_i$ and $1$ are trivial. To obtain $s_1+\sum_{j=1}^2\max\{\frac{t_j+s_{j+1}-s_1}{2}, 0\}$, take $A_i\subset[0,1]$ as in Theorem \ref{thm-our-example} with 
$$\gamma_{A_i}=0,\ \beta_{A_i}=s_i,$$
and $B_j\subset[0,1]$ as in Theorem \ref{thm-our-example} with $$\gamma_{B_j}=\max\{\frac{t_j+s_{j+1}-s_1}{2}, 0\}\ \beta_{B_j}=s_1-s_{j+1},$$
under the same sequence $\{q_i\}$. Then for every $b_j\in B_j$, the set $A_1+b_1A_2+b_2A_3$ is contained in the $q_i^{-1}$-neighborhood of
$$\frac{\Z}{q_i^{s_1}}+\frac{n_1}{H_1q_i^{s_1-s_2}}\cdot \frac{\Z}{q_i^{s_2}}+\frac{n_3}{H_2q_i^{s_1-s_3}}\cdot \frac{\Z}{q_i^{s_3}}\subset \frac{\Z}{H_1H_2 q_i^{s_1}}$$
in $[0,3]$, for some integers $1\leq H_j\leq q_i^{\frac{t_j+s_{j+1}-s_1}{2}}$. Hence it can be covered by no more than $$3\,q_i^{s_1+\sum_{j=1}^2\max\{\frac{t_j+s_{j+1}-s_1}{2}, 0\}}$$
intervals of length $q_i^{-1}$, as desired.
\end{proof}

\begin{proof}[Proof of Proposition \ref{thm-d=3-n=1}]
The upper bounds $\sum s_i$ and $1$ are trivial. To obtain $\frac{s_1+s_2+s_3+t}{3}$, take $A_i$ to be as Theorem \ref{thm-our-example} with
$$\gamma=0,\ \beta=s_i,$$
and $E\subset [0,1]^2$ as Theorem \ref{thm-our-example} with 
$$\gamma=\frac{t-(s_1-s_2)-(s_1-s_3)}{3},\ \beta_1=s_1-s_2,\ \beta_2=s_1-s_3,$$
under the same sequence $\{q_i\}$. Then lines determined by vectors in $\{1\}\times E$, namely
$$\Omega:=\{e\in S^1: \R e\cap (\{1\}\times E)\neq\emptyset\}$$
has Hausdorff dimension $\dH \Omega=\dH E=t$.

Similar to the previous subsection, for every $e\in \Omega$, the set $\pi_e(A_1\times A_2\times A_3)$ is contained in the $q_i^{-1}$-neighborhood of 
$$\frac{\Z}{q_i^{s_1}}+\frac{n_1}{Hq_i^{s_1-s_2}}\cdot \frac{\Z}{q_i^{s_2}}+\frac{n_2}{Hq_i^{s_1-s_3}}\cdot \frac{\Z}{q_i^{s_3}}\subset \frac{\Z}{Hq_i^{s_1}}$$
in $[0,3]$, for some integer $1\leq H\leq q_i^{\frac{t-(2s_1-s_2-s_3)}{3}}$. Hence it can be covered by no more than $$3Hq_i^{s_1}\leq 3q_i^{\frac{s_1+s_2+s_3+t}{3}}$$ many intervals of length $q_i^{-1}$, as desired.

Now we need to compare between $\frac{s_1+s_2+s_3+t}{3}$ and the upper bound in Proposition \ref{prop-prod-direction} with $t=t_1+t_2$. It turns out the only possibility that Proposition \ref{prop-prod-direction} wins is the case $t_2=1$. And the upper bound is the function $f$ as stated in the proposition.

\end{proof}

\section{Fourier dimension of Diophantine approximation}\label{sec-Fourier}
In this Section we prove Theorem \ref{thm-Fourier-dimension}. We may assume $\gamma>0$ as we have explained in the introduction that the case $\gamma=0$ has Fourier dimension zero and the Frostman measure has been constructed in Section \ref{subsec-gamma=0}. There will be three steps. First we construct a measure with desired Fourier decay, then we show no measure has faster Fourier decay, finally we show the measure from the first step also satisfies the Frostman condition. Some techniques are inspired by previous work. We refer to \cite{Wol03} for the classical
and \cite{FHR25} for a more recent version.

We shall prove a more general result than what Theorem \ref{thm-Fourier-dimension} is required. For $\kappa\in[0,\gamma]$, let
\begin{equation}
	\label{def-mathcal-P-i-kappa}
	\mathcal{P}_i^{\kappa,\gamma}:=\{\text{primes in }[q_i^\kappa/2, q_i^\gamma]\}.
\end{equation}
In particular, $\mathcal{P}_i^{\gamma,\gamma}=\mathcal{P}_i^\gamma$ defined in \eqref{def-mathcal-P-i}. This generalization is not necessary for Theorem \ref{thm-Fourier-dimension}, but it helps in the proof of Theorem \ref{thm-non-geo} in Section \ref{sec-dim-measure}. 

\subsection{Construct a measure with Fourier decay}\label{subsec-construct-a-measure}
We shall construct a nonzero finite Borel measure $\mu$ on 
$$\bigcap_i\bigcup_{p\in\mathcal{P}_i^{\kappa,\gamma}}\N_{q_i^{-1}}\left(\frac{\Z\backslash p\Z}{pq_i^{\beta}}\right)\cap[0,1]$$
with $|\hat{\mu}(k)|\lesssim_\epsilon |k|^{-\gamma+\epsilon}$, $|k|\in\Z\backslash\{0\}$. Then $|\hat{\mu}(\xi)|\lesssim_\epsilon |\xi|^{-\gamma+\epsilon}$ follows immediately (see, for example, Lemma 9.4.A in \cite{Wol03}). Notice that the decay exponent is independent in $\kappa$ and $\beta$.

Let $\phi\in C_0^\infty((-1,1))$ be nonnegative with $\int\phi=1$. Assume $q_i^\beta\in\Z$ for all $q_i$. For each prime $p$ we define
$\phi_{i,p}(x)$ as the ``modified" periodization of $p^{-1}q_i^{1-\beta}\phi(p^{-1}q_i^{1-\beta}x)$, that is
\begin{equation}\label{def-phi-i-p}\phi_{i,p}(x):=\sum_{v\in\Z\backslash p\Z}p^{-1}q_i^{1-\beta}\phi(p^{-1}q_i^{1-\beta}(x-v)).\end{equation}
In fact, for Fourier decay we do not have to exclude $p\Z$, but it is necessary for the Frostman condition. See Section \ref{subsec-Frostman} below.

Notice $\phi_{i,p}$ is $p$-periodic and has Fourier expansion $$\phi_{i,p}(x)=\sum_{n\in\Z}\hat{\phi}(pq_i^{\beta-1}n) e^{2\pi i nx}-p^{-1}\sum_{m\in\Z}\hat{\phi}(q_i^{\beta-1}m)e^{2\pi i mx/p}.$$ 
Rescale $\phi_{i,p}$ to
        \begin{equation}\label{def-Phi-i-p}\begin{aligned}\Phi_{i,p}(x):=&\phi_{i,p}(pq_i^\beta x)\\=&\sum_{v\in\mathbb{Z}\backslash p\Z}p^{-1}q_i^{1-\beta}\phi(q_i(x-\frac{v}{pq_i^\beta}))\\=&\sum_{n\in\Z}\hat{\phi}(pq_i^{\beta-1}n) e^{2\pi i pq_i^\beta n x}-p^{-1}\sum_{m\in\Z}\hat{\phi}(q_i^{\beta-1}m)e^{2\pi i q_i^\beta mx}.\end{aligned}\end{equation}
Then $\Phi_{i,p}$ is smooth on $\N_{q_i^{-1}}(\frac{\Z}{pq_i^\beta})$, $1$-periodic, and after restriction onto $[0,1]$ it has Fourier coefficients 
    \begin{align}\label{Fourier coefficient of phi}
        \widehat{\Phi_{i,p}}(k)=
        \begin{cases}
            (1-p^{-1})\hat{\phi}(q_i^{-1}k), & k\in pq_i^{\beta}\Z\\
            -p^{-1}\hat{\phi}(q_i^{-1}k), & k\in q_i^\beta\Z\backslash pq_i^\beta\Z\\
            0, & \mathrm{otherwise}
        \end{cases}.
    \end{align}
    In particular $\int_0^1\Phi_{i,p}=\widehat{\Phi_{i,p}}(0)=1-p^{-1}$.
    
    Then
    \begin{equation}\label{def-F-i}
        F_{i}(x):=\frac{1}{\#\mathcal{P}_i^{\kappa,\gamma}}\sum_{p\in\mathcal{P}_i^{\kappa,\gamma}}\frac{p}{p-1}\Phi_{i,p}(x)
    \end{equation}
    is smooth on \begin{equation}\label{supp-F-i}\bigcup_{p\in\mathcal{P}_i^{\kappa,\gamma}}\mathcal{N}_{q_i^{-1}}(\frac{\Z\backslash p\Z}{pq_i^\beta})\cap[0,1]\end{equation} and $\widehat{F_i}(0)=\int_0^1 F_i=1$. For $k\neq 0$, $\widehat{F_i}(k)$ equals
    $$
        (\#\mathcal{P}_i^{\kappa,\gamma})^{-1}\left(\#\{p\in\mathcal{P}_i^{\kappa,\gamma}: k\in pq_i^\beta\Z\}-\sum_{p\in\mathcal{P}_i^{\kappa,\gamma}: k\in q_i^\beta\Z\backslash pq_i^\beta\Z}\frac{1}{p-1}\right)\hat{\phi}(q_i^{-1}k).$$
    Then, by the prime number theorem, the trivial prime divisor bound
    $$\#\{p\in\mathcal{P}_i^{\kappa,\gamma}: pq_i^\beta\mid k\}\leq \frac{\log (|k|q_i^{-\beta})}{1+\log (q_i^\kappa/2)},$$
    and the fast decay of $\hat{\phi}(q_i^{-1}k)$, it follows that
    \begin{equation}\label{Fourier coefficient of Fi neq 0}\begin{aligned}
        |\widehat{F_i}(k)|\leq & C\frac{\log q_i^{\gamma}}{q_i^{\gamma}}\cdot\left(\frac{\log (|k| q_i^{-\beta})}{1+\log q_i^{\kappa}}+1\right)\cdot |\hat{\phi}(q_i^{-1}k)|\\\leq & C_{N}\frac{\log |k|+\log q_i}{q_i^{\gamma}}(1+\frac{|k|}{q_i})^{-N},\ |k|\neq 0.
    \end{aligned}
    \end{equation}
    Here and throughout this subsection, all constants $C, C_N, C_\phi$ may vary from line to line, may depend on $N, \kappa, \gamma, \beta, \phi, q_1$, but must be independent in $i,k$ and the choice of $q_i, i\geq 2$ under the condition 
    \begin{equation}\label{increasing-and-q-1-large}q_i>q_{i-1}^{10 i/\gamma} \text{ and } q_1>C_\phi.\end{equation}
     
    One can already see from \eqref{Fourier coefficient of Fi neq 0} that $\kappa$ and $\beta$ make no contribution to the Fourier decay exponent.

    Although $F_i$ seems to have desired support and Fourier decay, its weak limit is the Lebesgue measure due to $\widehat{F_i}\to \delta_0$. To overcome this difficulty, we need to take their product. The key Lemma is the following. 

\begin{lem}\label{stability}
Suppose $\psi\in C^\infty([0,1])$. Then
        $$|\widehat{\psi F_i}(k)-\widehat{\psi}(k)|\leq C \|\psi\|\cdot
        \begin{cases}
            q_i^{-\gamma}\log q_i, & |k|\leq q_i\\
            |k|^{-\gamma}\log|k|, & |k|\geq q_i
        \end{cases},$$
    where in this subsection \begin{equation}\label{norm-psi-1}\|\psi\|:=|\hat{\psi}(0)|+\sum|\hat{\psi}(l)||l|^\gamma.\end{equation}
\end{lem}
Notice $\|\psi\|\leq \|\psi\|_{L^\infty}+(2\pi)^2\|\psi''\|_{L^\infty}\sum_{l\in\Z\backslash\{ 0\}}|l|^{-2+\gamma}$ for later use.
\begin{proof}[Proof of Lemma \ref{stability}]
    As $\widehat{F_i}(0)=1$,
    \begin{align*}
        \widehat{\psi F_i}(k)-\widehat{\psi}(k)=\sum_{l\in\mathbb{Z}}\widehat{\psi}(k-l)\widehat{F_i}(l)-\widehat{\psi}(k)
        =\sum_{l\neq0}\widehat{\psi}(k-l)\widehat{F_i}(l).
    \end{align*}
    
    When $|k|\leq q_i$, by \eqref{Fourier coefficient of Fi neq 0} with $|\hat{\phi}|\leq 1$,
    $$\left|\sum_{l\neq0}\hat{\psi}(k-l)\widehat{F_i}(l)\right|\leq C \sum_{l\neq 0}|\hat{\psi}(k-l)|\cdot \frac{\log q_i+\log|l|}{q_i^\gamma}.$$
    For $|k-l|> |l|/2$, it is
    $$\leq C\sum_{|k-l|\neq 0}|\hat{\psi}(k-l)|\cdot\frac{\log q_i+\log 2|k-l|}{q_i^\gamma}\leq C\|\psi\|\cdot q_i^{-\gamma}\log q_i.$$
    For $|k-l|\leq |l|/2$, due to $|l|\approx |k|\leq q_i$ it is
    $$\leq C\|\hat{\psi}\|_{l^1}\cdot q_i^{-\gamma}\log q_i\leq C\|\psi\|\cdot q_i^{-\gamma}\log q_i.$$
    
    From now we assume $|k|\geq q_i$ and write this sum as
    $$\sum_{l\neq 0: |k-l|>|k|/2}+\sum_{l\neq 0: |k-l|\leq |k|/2}:=I+II.$$
    It is easy to estimate $I$: in this case $1\leq 2 |k|^{-1}|k-l|$, so, by $|\widehat{F_i}|\leq 1$,
    $$\sum_{l\neq0: |k-l|>|k|/2}|\widehat{\psi}(k-l)||\widehat{F_i}(l)|\leq \sum_{|k-l|>|k|/2}|\widehat{\psi}(k-l)|\leq  C\|\psi\|\cdot |k|^{-\gamma}.
    $$
    For $II$, in this case $q_i/2<|k|/2\leq |l|\leq 3|k|/2$, so by \eqref{Fourier coefficient of Fi neq 0} with $N=\gamma$,
    $$\begin{aligned}
        &\sum_{l\neq0: |k-l|>|k|/2}|\widehat{\psi}(k-l)||\widehat{F_i}(l)|\\\leq &C\sum_{|k|/2\leq |l|\leq 3|k|/2}|\widehat{\psi}(k-l)| \cdot\frac{\log|k|}{q_i^\gamma}\cdot(\frac{|l|}{q_i})^{-\gamma}\\\leq &C \|\hat{\psi}\|_{l^1(\Z)}\cdot |k|^{-\gamma}\log|k|\leq C\|\psi\|\cdot |k|^{-\gamma}\log|k|.
    \end{aligned}$$
\end{proof}

   Now we take $G_0=\chi_{[0,1]}$ and $G_m=\prod_{i=1}^mF_i$, with
   \begin{equation}\label{supp-G-m}\supp (G_m)\subset \bigcap_{i=1}^m\bigcup_{p\in\mathcal{P}_i^{\kappa,\gamma}}\mathcal{N}_{q_i^{-1}}(\frac{\Z\backslash p\Z}{pq_i^\beta})\cap[0,1]\end{equation}
   due to the support of $F_i$ in \eqref{supp-F-i}. 
   
   Applying Lemma \ref{stability} with $\psi=G_m$, we have, for all $m\geq 0$,
   \begin{equation}\label{apply-stability}
       |\widehat{G_{m+1}}(k)-\widehat{G_m}(k)|\leq C
        \|G_m\|\cdot \begin{cases} q_{m+1}^{-\gamma}\log q_{m+1}, & |k|\leq q_{m+1}\\
            |k|^{-\gamma}\log|k|, & |k|\geq q_{m+1}
        \end{cases},
   \end{equation}
   where $\|G_m\|$ is defined as in \eqref{norm-psi-1}. By our construction of $F_i$ in \eqref{def-F-i} and \eqref{def-Phi-i-p}, 
    \begin{equation}\label{L-infty-F-i}\begin{aligned}\|F_i\|_{L^\infty}\leq &\max_{p\in\mathcal{P}_i^{\kappa,\gamma}}\|\phi_{i,p}\|_{L^\infty}\leq 2\|\phi\|_{L^\infty}\cdot q_i^{1-\kappa-\beta},\\\|F_i''\|_{L^\infty}\leq &\max_{p\in\mathcal{P}_i^{\kappa,\gamma}}\|\phi_{i,p}''\|_{L^\infty}\leq 2\|\phi''\|_{L^\infty}\cdot q_i^{3-\kappa-\beta}.\end{aligned}\end{equation}
    Therefore for all $m\geq 1$,
    $$\|G_m\|=\|\prod_{i=1}^m F_i\|\leq \|(\prod_{i=1}^m F_i)\|_{L^\infty}+C\|(\prod_{i=1}^m F_i)''\|_{L^\infty}\leq m^2 C_\phi^m\prod_{i=1}^m q_i^{3-\kappa-\beta},$$
   which is $\leq q_{m+1}^{\min\{\frac{1}{m}, \frac{\gamma}{2}\}}$ if $q_i$ is increasing rapidly satisfying \eqref{increasing-and-q-1-large}. 
   Then \eqref{apply-stability} becomes
   \begin{equation}\label{hat-G-m+1-G-m}|\widehat{G_{m+1}}(k)-\widehat{G_m}(k)|\leq C
        \begin{cases} q_{m+1}^{-\gamma+\min\{\frac{1}{m}, \frac{\gamma}{2}\}}\log q_{m+1}, & |k|\leq q_{m+1}\\
            q_{m+1}^{\min\{\frac{1}{m}, \frac{\gamma}{2}\}} |k|^{-\gamma}\log|k|, & |k|\geq q_{m+1}
        \end{cases},\end{equation}
        where the constant $C$ is independent on $k, m$, and the choice of $q_i$ under \eqref{increasing-and-q-1-large}.

        With \eqref{hat-G-m+1-G-m} in hand we can construct a desired measure $\mu$.
    
    First, as $\widehat{G_1}(0)=\widehat{F_1}(0)=1$,
     $$|\widehat{G_{m+1}}(0)-1|\leq\sum_{i=1}^m |\widehat{G_{i+1}}(0)-\widehat{G_i}(0)|\leq C \sum_{i=1}^\infty q_{i+1}^{-\gamma+\min\{\frac{1}{i}, \frac{\gamma}{2}\}}\log q_{i+1}.$$
     As $C$ is uniform for any sequence $\{q_i\}$ satisfying \eqref{increasing-and-q-1-large}, the right hand side is $<1/2$ when $q_1$ is large enough, which implies
     $$1/2\leq |\widehat{G_m}(0)|=|\int_0^1 G_m|\leq 3/2.$$
     Consequently there exists a subsequence $G_{m_j}$ whose weak limit $\mu=\lim G_{m_j}$ is a nonzero finite Borel measure on $[0,1]$. As $\supp (G_m)$ in \eqref{supp-G-m} is decreasing in $m$,
     $$\supp\mu\subset\lim_{m\rightarrow\infty}\supp (G_m)\subset \bigcap_i\bigcup_{p\in\mathcal{P}_i^{\kappa,\gamma}}\N_{q_i^{-1}}\left(\frac{\Z\backslash p\Z}{pq_i^{\beta}}\right).$$
     In fact, $\mu$ is the weak limit of $G_m$ because $G_m$ is nonnegative and $\{\widehat{G_m}(k)\}_m$ is a Cauchy sequence for every $k$. More precisely, $\mu=\lim G_{m_j}$ and $\{\widehat{G_m}(k)\}_m$ is Cauchy imply
     $$\hat{\mu}(k)=\lim_{m\rightarrow\infty} \widehat{G_m}(k),\ \forall\,k\in\Z,$$
     and therefore for every $1$-periodic smooth function $\phi$,
     $$\int \phi\,d\mu=\sum_{k\in\Z}\hat{\phi}(k)\overline{\hat{\mu}(k)}=\lim_{m\rightarrow\infty} \sum_{k\in\Z}\hat{\phi}(k)\overline{\widehat{G_m}(k)}=\lim_{m\rightarrow\infty} \int \phi(x)G_m(x)\,dx.$$
     Finally, for every $1$-periodic continuous function $f$, for every $\epsilon>0$ there exists a $1$-periodic smooth $\phi$ such that $|f-\phi|\leq \epsilon$. Hence
     $$\left|\int f\,d\mu-\int f(x)G_m(x)\,dx\right|$$
     $$\leq \int |f-\phi|\,d\mu+\left|\int \phi\,d\mu-\int \phi(x)G_m(x)\,dx\right|+\int |f(x)-\phi(x)|G_m(x)dx$$
     whose $\limsup$ is $\leq 4\epsilon$ as $m\rightarrow\infty$.
     
     It remains to show $$|\hat{\mu}(k)|\lesssim_{\epsilon, \{q_i\}} |k|^{-\gamma+\epsilon}, \,k\neq 0.$$ For every $q_{m+1}\geq |k|$ we write
     $$|\widehat{G_{m+1}}(k)|=|\widehat{G_{m+1}}(k)-\widehat{G_0}(k)|\leq\sum_{i=0}^m |\widehat{G_{i+1}}(k)-\widehat{G_i}(k)|=\sum_{q_{i+1}\leq |k|}+\sum_{q_{i+1}\geq|k|}$$
     and by \eqref{hat-G-m+1-G-m} it is
     $$\leq C|k|^{-\gamma}\log|k|\sum_{q_{i+1}\leq |k|}q_{i+1}^{\frac{1}{i}}+C\sum_{q_{i+1}\geq |k|}q_{i+1}^{-\gamma+\frac{1}{i}}\log q_{i+1}\lesssim_{\epsilon, \{q_i\}}|k|^{-\gamma+\epsilon}$$
     as $q_i$ increases rapidly. This completes the first step of the proof of Theorem \ref{thm-Fourier-dimension}.
\subsection{No measure has faster Fourier decay}\label{sub-contradiction-on-Fourier-decay}
We may assume $\beta>0$, otherwise it is trivial because of $\dF\leq \dH$. Let the sequence $\{q_i\}$ be as the previous subsection. Suppose there exists a finite Borel measure $\mu$ supported on
$$E:=\begin{cases}
	    \bigcap_i\bigcup_{1\leq H\leq q_i^\gamma}\N_{q_i^{-1}}\left(\frac{\Z}{Hq_i^{\beta}}\right), & \text{if } 2\gamma+\beta<1\\
     \bigcap_i\bigcup_{1\leq H\leq q_i^\gamma, prime}\N_{q_i^{-1}}\left(\frac{\Z}{Hq_i^{\beta}}\right), & \text{if } 2\gamma+\beta=1
	\end{cases}$$
 with
 $$|\hat{\mu}(\xi)|\lesssim |\xi|^{-\gamma'}$$
 for some $\gamma'>\gamma$. We shall find a subsequence $q_{i_j}$ and construct a measure $\nu$ supported on
$$E':=\begin{cases}
	    \bigcap_j\bigcup_{1\leq H\leq q_{i_j}^\gamma}\N_{2q_{i_j}^{-(1-\beta)}}\left(\frac{\Z}{H}\right), & \text{if } 2\gamma+\beta<1\\
     \bigcap_j\bigcup_{1\leq H\leq q_{i_j}^\gamma, prime}\N_{2q_{i_j}^{-(1-\beta)}}\left(\frac{\Z}{H}\right), & \text{if } 2\gamma+\beta=1
	\end{cases}$$
 satisfying
 $$|\hat{\nu}(\xi)|\lesssim |\xi|^{-\frac{\gamma'}{1-\beta}}.$$
This is absurd: when $2\gamma+\beta<1$ this implies $\dF E'\geq\min\{\frac{2\gamma'}{1-\beta}, 1\}>\frac{2\gamma}{1-\beta}=\dH E'$; when $2\gamma+\beta=1$ this implies $E'$ has positive Lebesgue measure. Both are contradictions.

Now we construct $\nu$. We may assume $\supp\mu\subset(0,1)$ as a smooth cutoff preserves Fourier decay. Also we only deal with the case $2\gamma+\beta<1$ because there is no difference for $2\gamma+\beta=1$ in this step.

Let $\phi\in C_0^\infty([-1,1])$, nonnegative and $\int\phi=1$. Denote $\phi_i(x):=q_i^{-1}\phi(q_ix)$. First consider the $q_i^{-1}$-localizaton of $\mu$, i.e. $\mu*\phi_i$ supported on $[0,1]$. Then rescale it to $[0, q_i^\beta]$, i.e. $q_i^{-\beta}\mu*\phi_i(q_i^{-\beta}\cdot)$. Finally take $F_i$ to be its $1$-periodization, i.e.
$$F_i(x)=\sum_{v\in\Z}q_i^{-\beta}\mu*\phi_i(q_i^{-\beta}(x-v)).$$
Then $F_i$ is $1$-periodic,
$$\supp F_i\subset\bigcup_{1\leq H\leq q_i^\gamma}\N_{2q_i^{-(1-\beta)}}\left(\frac{\Z}{H}\right),$$
and it is straightforward to check that after restriction onto $[0,1]$ its Fourier coefficients are
$$\widehat{\mu*\phi_i}(q_i^\beta k)=\hat{\mu}(q_i^\beta k)\hat{\phi}(q_i^{-1+\beta} k).$$
In particular we have 
\begin{equation}\label{hat-F-i-0-sec-4.2}\widehat{F_i}(0)=1\end{equation}
and for $|k|\neq 0$,
\begin{equation}\label{hat-F-i-neq-0-sec-4.2}|\widehat{F_i}(k)|\leq C_\mu (q_i^{\beta} |k|)^{-\gamma'}|\hat{\phi}(q_i^{-1+\beta}k)|\leq C_{\mu, N}\, (q_i^{\beta} |k|)^{-\gamma'}(1+\frac{|k|}{q_i^{1-\beta}})^{-N}\end{equation}
by the given Fourier decay of $\mu$ and the fast decay of $\hat{\phi}$. Here and throughout this subsection, all constants $C, C_\mu, C_{\mu, N}$ may vary from line to line, may depend on $\mu, N, \gamma', \beta, \phi, q_{i_1}$, but must be independent in $i,k$ and the choice of $q_{i_j}, j\geq 2$ satisfying
\begin{equation}\label{increasing-and-q-1-large-sec-4.2}q_{i_{j+1}}>q_{i_j}^{10jC_{\gamma',\beta}},\ q_{i_1}>C_0.\end{equation}

Now the situation is quite similar to the previous subsection: $F_i$ has desired support and desired Fourier decay, while $F_i\rightarrow\delta_0$. So again we take their product. The key lemma analogous to Lemma \ref{stability} is the following.

\begin{lem}\label{stability-lem-4.2}
    Suppose $\psi\in C^\infty([0,1])$. Then
    \begin{align}
        |\widehat{\psi F_i}(k)-\widehat{\psi}(k)|\leq C \|\psi\|\cdot
        \begin{cases}
            q_i^{-\beta\gamma'}(1+|k|)^{-\gamma'}, & |k|\leq q_i^{1-\beta}\\
            |k|^{-\frac{\gamma'}{1-\beta}}, & |k|\geq q_i^{1-\beta}
        \end{cases},
    \end{align} 
    where in this subsection
    \begin{equation}\label{norm-psi-2}\|\psi\|:=|\hat{\psi}(0)|+\sum|\hat{\psi}(l)||l|^\frac{\gamma'}{1-\beta}.\end{equation}
\end{lem}
The proof is similar to Lemma \ref{stability} but one needs to be careful because the behavior of $\widehat{F_i}$ is not the same. 
\begin{proof}[Proof of Lemma \ref{stability-lem-4.2}]
The first step is again to write
\begin{align*}
        \widehat{\psi F_i}(k)-\widehat{\psi}(k)=\sum_{l\in\mathbb{Z}}\widehat{\psi}(k-l)\widehat{F_i}(l)-\widehat{\psi}(k)
        =\sum_{l\neq0}\widehat{\psi}(k-l)\widehat{F_i}(l).
    \end{align*}
Then we directly split the sum into
$$\sum_{l\neq 0: |k-l|>|k|/2}+\sum_{l\neq 0: |k-l|\leq |k|/2}:=I+II.$$

For $I$: in this case $1\leq 4 (1+|k|)^{-1}|k-l|$. By \eqref{hat-F-i-neq-0-sec-4.2} with $|\hat{\phi}|\leq 1$ we have $|\widehat{F_i}(l)|\leq C (q_i^\beta l)^{-\gamma'}\leq C q_i^{-\beta\gamma'}$ for all $l\neq 0$. Therefore
    $$\begin{aligned}&\sum_{l\neq0: |k-l|>|k|/2}|\widehat{\psi}(k-l)||\widehat{F_i}(l)|\\\leq & C q_i^{-\beta\gamma'}\sum_{|k-l|>|k|/2}|\widehat{\psi}(k-l)|\\ \leq & C q_i^{-\beta\gamma'}(1+|k|)^{-\frac{\gamma'}{1-\beta}}\sum_{|k-l|>|k|/2}|\widehat{\psi}(k-l)||k-l|^{\frac{\gamma'}{1-\beta}}\\ \leq & C\|\psi\|\cdot q_i^{-\beta\gamma'}(1+|k|)^{-\frac{\gamma'}{1-\beta}},\end{aligned}
    $$
desired for both $|k|\leq q_i^{1-\beta}$ and $|k|\geq q_i^{1-\beta}$. This also settles the case $k=0$.

For $II$, in this case $0<|k|/2\leq |l|\leq 3|k|/2$. When $0<|k|\leq q_i^{1-\beta}$, by \eqref{hat-F-i-neq-0-sec-4.2} with $|\hat{\phi}|\leq 1$ we have
$$\sum_{l\neq0: |k-l|\leq|k|/2}|\widehat{\psi}(k-l)||\widehat{F_i}(l)|\leq C q_i^{-\beta\gamma'}\sum_{|l|\approx|k|}|\widehat{\psi}(k-l)||l|^{-\gamma'}\leq C\|\psi\| q_i^{-\beta\gamma'}|k|^{-\gamma'}.$$
When $|k|\geq q_i^{1-\beta}$, by \eqref{hat-F-i-neq-0-sec-4.2} with $N=\frac{\beta\gamma'}{1-\beta}$ we have
$$\begin{aligned}\sum_{l\neq0: |k-l|\leq|k|/2}|\widehat{\psi}(k-l)||\widehat{F_i}(l)|\leq C  & \sum_{|l|\approx|k|}|\widehat{\psi}(k-l)|q_i^{-\beta\gamma'}|l|^{-\gamma'}(\frac{|k|}{q_i^{1-\beta}})^{-\frac{\beta\gamma'}{1-\beta}}\\\leq C &\|\psi\|\cdot |k|^{-\frac{\gamma'}{1-\beta}}.\end{aligned}$$
\end{proof}

Now, let $q_{i_j}$ be a subsequence of $q_i$, take $G_0=\chi_{[0,1]}$ and $G_m=\prod_{j=1}^mF_{i_j}$. By Lemma \ref{stability-lem-4.2} with $\psi=G_m$, we have, for all $m\geq 0$,
   \begin{equation}\label{apply-stability-lem-4.2}
       |\widehat{G_{m+1}}(k)-\widehat{G_m}(k)|\leq C
        \|G_m\|\cdot \begin{cases} q_{i_{m+1}}^{-\beta\gamma'}(1+|k|)^{-\gamma'}, & |k|\leq q^{1-\beta}_{i_{m+1}}\\
            |k|^{-\frac{\gamma'}{1-\beta}}, & |k|\geq q^{1-\beta}_{i_{m+1}}
        \end{cases},
   \end{equation}
   where $\|G_m\|$ is defined as in \eqref{norm-psi-2}. 
   
   Then, similar to the proof of Lemma \ref{stability},
    $$\|G_m\|=\|\prod_{j=1}^m F_{i_j}\|\leq \|(\prod_{i=j}^m F_i)\|_{L^\infty}+C\|\partial^{[\frac{\gamma'}{1+\beta}]+2}(\prod_{i=j}^m F_{i_j})\|_{L^\infty}\leq C_0^m\prod_{j=1}^m q_{i_j}^{C_{\gamma',\beta}}.$$
    As $C_0$ is independent in the choice of $q_{i_j}$,
    $$\|G_m\|\leq q_{i_{m+1}}^{\min\{\frac{1}{m}, \frac{\beta\gamma'}{2}\}},\ \forall\,m\geq 1,$$
    for every choice of $q_{i_j}$ satisfying \eqref{increasing-and-q-1-large-sec-4.2}.
    
Then \eqref{apply-stability-lem-4.2} becomes
   \begin{equation}\label{hat-G-m+1-G-m-sec-4.2}
       |\widehat{G_{m+1}}(k)-\widehat{G_m}(k)|\leq C
        \begin{cases} q_{i_{m+1}}^{-\beta\gamma'+\min\{\frac{1}{m},\frac{\beta\gamma'}{2}\}}(1+|k|)^{-\gamma'}, & |k|\leq q^{1-\beta}_{i_{m+1}}\\
            q_{i_{m+1}}^{\frac{1}{m}}|k|^{-\frac{\gamma'}{1-\beta}}, & |k|\geq q^{1-\beta}_{i_{m+1}}
        \end{cases},
   \end{equation}
   where the constant $C$ is uniform in any sequence $\{q_{i_j}\}$ satisfying \eqref{increasing-and-q-1-large-sec-4.2}.

   The rest is nothing different from the previous subsection. So we omit some details. First, since
     $$|\widehat{G_{m+1}}(0)-1|\leq\sum_{i=1}^m |\widehat{G_{i+1}}(0)-\widehat{G_i}(0)|\leq C \sum_{i=1}^\infty q_{i+1}^{-\beta\gamma'+\min\{\frac{1}{i}, \frac{\beta\gamma'}{2}\}},$$
   we can choose $\{q_{i_j}\}$ properly to ensure the existence of a subsequence of $G_m$ whose weak limit $\nu$ is nontrivial and supported on $E'$ as expected. To see its Fourier decay, for every $q^{1-\beta}_{i_{m+1}}\geq |k|$ we write
   $$|\widehat{G_{m+1}}(k)|=|\widehat{G_{m+1}}(k)-\widehat{G_0}(k)|\leq\sum_{j=0}^m |\widehat{G_{j+1}}(k)-\widehat{G_j}(k)|=\sum_{q^{1-\beta}_{i_{j+1}}\leq |k|}+\sum_{q^{1-\beta}_{i_{j+1}}\geq|k|}$$
     and by \eqref{hat-G-m+1-G-m-sec-4.2} it is
     $$\leq C\left(|k|^{-\frac{\gamma'}{1-\beta}}\sum_{q^{1-\beta}_{i_{j+1}}\leq |k|}q_{i_{j+1}}^{\frac{1}{j}}\right)+C\left((1+|k|)^{-\gamma'}\sum_{q^{1-\beta}_{i_{j+1}}\geq |k|}q_{i_{j+1}}^{-\beta\gamma'+\frac{1}{j}}\right)$$ $$\lesssim_{\epsilon, \{q_{i_j}\}}|k|^{-\frac{\gamma'}{1-\beta}+\epsilon},$$
     as desired. This completes the second step of the proof of Theorem \ref{thm-Fourier-dimension}.

\subsection{The Frostman condition}\label{subsec-Frostman}

Denote 
$$s:=\kappa+\gamma+\beta.$$
In this subsection we show the measure $\mu$ constructed in Section \ref{subsec-construct-a-measure} satisfies the Frostman condition
\begin{equation}
\label{Frostman-kappa}
	\mu(B(x,r))\lesssim_{\epsilon} r^{s-\epsilon},\ \forall\,x\in\R, r>0.
\end{equation}
For Theorem \ref{thm-Fourier-dimension}, it suffices to take $\kappa=\gamma$. A general $\kappa\in(0,\gamma]$ shall help in the proof of Theorem \ref{thm-non-geo} in Section \ref{sec-dim-measure}.

Fix a $\psi\in C_0^\infty((-1,1))$, nonnegative. It suffices to prove
$$\int \psi(\frac{x-y}{r})\,d\mu(y)\lesssim_\epsilon r^{s-\epsilon},\ \forall\,x\in\R, r>0.$$
Notice that
\begin{equation}\label{ball-condition-Fourier-side}\int \psi(\frac{x-y}{r})\,d\mu(y)=\int e^{2\pi ix\xi }r\hat{\psi}(r\xi)\hat{\mu}(\xi)\,d\xi.\end{equation}

Recall $\mu$ is the weak limit of $G_m$. We claim that when $q_m\geq r^{-1}$, the difference between \eqref{ball-condition-Fourier-side} and 
$$\int e^{2\pi ix\xi }r\hat{\psi}(r\xi)\widehat{G_m}(\xi)\,d\xi$$
is negligible. To see this, consider
\begin{equation}\label{sum-m-infty}\int r|\hat{\psi}(r\xi)||\widehat{\mu}(\xi)-\widehat{G_{m}}(\xi)|\,d\xi=\lim_{m'\rightarrow\infty}\int r|\hat{\psi}(r\xi)||\widehat{G_{m'}}(\xi)-\widehat{G_{m}}(\xi)|\,d\xi.\end{equation}
For $m'>m$, write
$$\int r|\hat{\psi}(r\xi)||\widehat{G_{m'}}(\xi)-\widehat{G_{m}}(\xi)|\,d\xi\leq \sum_{j=m}^\infty \int r|\hat{\psi}(r\xi)||\widehat{G_{j+1}}(\xi)-\widehat{G_j}(\xi)|\,d\xi.$$
By \eqref{hat-G-m+1-G-m}, $\|G_{j+1}-G_j\|_{L^\infty}\leq Cq_{i+1}^{-\gamma+\min\{\frac{1}{m}, \frac{\gamma}{2}\}}\log q_i$. Therefore \eqref{sum-m-infty} is
$$\leq C\left(\sum_{j\geq m}q_{j+1}^{-\gamma+\min\{\frac{1}{m}, \frac{\gamma}{2}\}}\log q_{j+1}\right)\int r|\hat{\psi}(r\xi)|\,d\xi\leq q_m^{-1}<r\ll r^{s-\epsilon}$$
when $q_i$ is increasing fast enough.

From the discussion above it suffices to show
$$\int_{B(x,r)} G_{m_0}(y)\,dy\lesssim_\epsilon r^{s-\epsilon}$$
for some $m_0$ with $q_{m_0}\geq r^{-1}$ that will be chosen later. 

We need to estimate $|G_m(y)|$. The estimate \eqref{L-infty-F-i} implies that
$$|G_{m}(y)|=|\prod_{i=1}^{m}F_i(y)|\leq C_\phi^{m-2}\left(\prod_{i=1}^{m-2} q_i^{1-\kappa-\beta}\right)\cdot F_{m-1}(y)\cdot F_m(y),$$
while we need a more careful estimate on $F_{m-1}, F_m$ than \eqref{L-infty-F-i}. 

Thanks to the exclusion of $p\Z$ 
 in the definition of $\phi_{i,p}$ in \eqref{def-phi-i-p}, the separation 
 \begin{equation}
        \label{seperation-kappa}
        \left|\frac{m}{pq_i^\beta}-\frac{m'}{p'q_i^\beta}\right|\geq q_i^{-2\kappa-\beta}\geq q_i^{-\kappa-\gamma-\beta}=q_i^{-s}, \ \forall\,(m,p)\neq(m',p')
    \end{equation}
 implies that the supports of $\phi_{i,p}$ are disjoint between different $p\in \mathcal{P}_i^{\kappa,\gamma}$. Therefore, by the definition of $\Phi_{i,p}$ in \eqref{def-Phi-i-p} and $F_i$ in \eqref{def-F-i}, one can conclude that
 \begin{equation}\label{upper-bound-F-i-kappa}|F_i(y)|\leq \sum_{p\in\mathcal{P}_i^{\kappa,\gamma}} \frac{2\|\phi\|_{L^\infty}p^{-1}q_i^{1-\beta}}{\#\mathcal{P}_i^{\kappa,\gamma}}\cdot\chi_{\mathcal{N}_{q_i^{-1}}\left(\frac{\Z\backslash p\Z}{pq_i^\beta}\right)},\end{equation}
 where functions in the sum have disjoint supports. In particular,
 \begin{equation}
 	\label{L-infty-F-i-kappa}
 	\|F_i\|_{L^\infty}\leq C_\phi q_i^{1-\kappa-\beta}/\#\mathcal{P}_i^{\kappa,\gamma},
 \end{equation}
 which is much better than \eqref{L-infty-F-i}.
 
 Then, with \eqref{L-infty-F-i-kappa} on $F_{m-1}$ and \eqref{upper-bound-F-i-kappa} on $F_m$,
\begin{equation}\label{reduction-on-G-m-kappa}
\begin{aligned}
    &\int_{B(x,r)} G_{m}(y)\,dy\\\leq  & C_\phi^{m-2}\left(\prod_{i=1}^{m-2} q_i^{1-\kappa-\beta}\right)\|F_{m-1}\|_{L^\infty}\int_{B(x,r)} F_m(y)\,dy\\\leq & C_\phi^{m}\left(\prod_{i=1}^{m-1} q_i^{1-\kappa-\beta}\right)\frac{1}{\#\mathcal{P}_{m-1}^{\kappa,\gamma}}\int_{B(x,r)}\sum_{p\in\mathcal{P}_m^{\kappa,\gamma}} \frac{p^{-1}q_m^{1-\beta}}{\#\mathcal{P}_m^{\kappa,\gamma}}\cdot\chi_{\mathcal{N}_{q_m^{-1}}\left(\frac{\Z\backslash p\Z}{pq_m^\beta}\right)}\\=&C_\phi^{m}\left(\prod_{i=1}^{m-1} q_i^{1-\kappa-\beta}\right) \frac{1}{\#\mathcal{P}_{m-1}^{\kappa,\gamma}}\sum_{p\in\mathcal{P}_m^{\kappa,\gamma}}\frac{p^{-1}q_m^{1-\beta}}{\#\mathcal{P}_m^{\kappa,\gamma}} \left|B(x,r)\cap \mathcal{N}_{q_m^{-1}}\left(\frac{\Z\backslash p\Z}{pq_m^\beta}\right)\right|.\end{aligned}
\end{equation}

Now choose $m_0$ such that $$q_{m_0}^{-s}\leq r<q_{m_0-1}^{-s},$$
with $q_0:=1$ as convention. This is compatible with our earlier assumption $q_{m_0}\geq r^{-1}$.

There are two ways to estimate 
$$\int_{B(x,r)} G_{m_0}(y)\,dy.$$

First, by simple counting, the number of $q_{m_0}^{-1}$-intervals contained in 
$$B(x,r)\cap \mathcal{N}_{q_{m_0}^{-1}}\left(\frac{\Z\backslash p\Z}{pq_{m_0}^\beta}\right)$$ is at most $rpq_{m_0}^\beta$. Therefore one can apply \eqref{reduction-on-G-m-kappa} with $m=m_0$ to obtain
\begin{equation}
    \label{estimate-B-x-r-G-m-0-1}
    \begin{aligned}
        & \int_{B(x,r)} G_{m_0}(y)\,dy\\\leq &  C_\phi^{m_0}\left(\prod_{i=1}^{m_0-1} q_i^{1-\kappa-\beta}\right)\frac{1}{\#\mathcal{P}_{m_0-1}^{\kappa,\gamma}} \sum_{p\in\mathcal{P}_{m_0}^{\kappa,\gamma}}\frac{p^{-1}q_{m_0}^{1-\beta}}{\#\mathcal{P}_{m_0}^{\kappa,\gamma}}\cdot rpq_{m_0}^{\beta-1}\\= &C_\phi^{m_0}\left(\prod_{i=1}^{m_0-1} q_i^{1-\kappa-\beta}\right)\cdot \frac{1}{\#\mathcal{P}_{m_0-1}^{\kappa,\gamma}}\cdot r.
    \end{aligned}
\end{equation}

Second, as $r<q_{m_0-1}^{-s}$, the separation \eqref{seperation-kappa} implies that $B(x,r)$ intersects at most one $q_{m_0-1}^{-1}$-interval, lying near $p_0^{-1}q_{m_0-1}^{-\beta}\Z$ for some $p_0\in \mathcal{P}_{m_0-1}^{\kappa,\gamma}$. As the difference between 
$$\int_{B(x,q_{m_0-1}^{-1})} G_{m_0}(y)\,dy \quad \text{and} \quad \int_{B(x,q_{m_0-1}^{-1})} G_{m_0-1}(y)\,dy$$
is negligible by the discussion at the beginning of this subsection, we can apply \eqref{reduction-on-G-m-kappa} with $m=m_0-1$ to obtain
\begin{equation}
    \label{estimate-B-x-r-G-m-0-2}
    \begin{aligned}& \int_{B(x,r)} G_{m_0}(y)\,dy\sim \int_{B(x,r)} G_{m_0-1}(y)\,dy \\ \leq & C_\phi^{m_0-1}\left(\prod_{i=1}^{m_0-2} q_i^{1-\kappa-\beta}\right)\cdot \frac{1}{\#\mathcal{P}_{m_0-2}^{\kappa,\gamma}}\cdot \sum_{p=p_0}\frac{p^{-1}q_{m_0-1}^{1-\beta}}{\#\mathcal{P}_{m_0-1}^{\kappa,\gamma}}\cdot q_{m_0-1}^{-1}\\\leq &  C_\phi^{m_0}\left(\prod_{i=1}^{m_0-1} q_i^{1-\kappa-\beta}\right)\cdot \frac{1}{\#\mathcal{P}_{m_0-1}^{\kappa,\gamma}}\cdot q_{m_0-1}^{-1}.
    \end{aligned}
\end{equation}
    
Finally we interpolate between \eqref{estimate-B-x-r-G-m-0-1} and \eqref{estimate-B-x-r-G-m-0-2} with $s=\kappa+\gamma+\beta$ to obtain
$$\begin{aligned}
    \int_{B(x,r)} G_{m_0}(y)\,dy\leq & C_\phi^{m_0}\left(\prod_{i=1}^{m_0-1} q_i^{1-\kappa-\beta}\right)\cdot \frac{1}{\#\mathcal{P}_{m_0-1}^{\kappa,\gamma}}\cdot r^{s-\frac{1}{m_0}}q_{m_0-1}^{-(1-s+\frac{1}{m_0})}\\ \leq & C_\phi^{m_0}\left(\prod_{i=1}^{m_0-2} q_i^{1-\kappa-\beta}\right) q_{m_0-1}^{-\frac{1}{m_0}}\log q_{m_0-1}\cdot r^{s-\frac{1}{m_0}},
\end{aligned}$$
which is $\lesssim_\epsilon r^{s-\epsilon}$ when $q_i$ rapidly increases. This finishes the proof of \eqref{Frostman-kappa}.

By taking $\kappa=\gamma$, \eqref{Frostman-kappa} becomes
$$\mu(B(x,r))\lesssim_\epsilon r^{2\gamma+\beta-\epsilon},$$
which completes the last step of the proof of Theorem \ref{thm-Fourier-dimension} because $\dH \supp\mu=2\gamma+\beta$ by Theorem \ref{thm-our-example}.

\section{Fourier restriction: the geometric case}\label{sec-Sharpness}

In this section, we prove Theorem \ref{thm-geo}. The set $E$ and measure $\mu$ are directly from Theorem \ref{thm-Fourier-dimension} with $2\gamma+\beta=a$ and $2\gamma=b$. Then it remains to show $(ii)$ in Theorem \ref{thm-geo}.

    As $p$ also represents the prime, in this section we use $\tilde{p}$ for the $L^{\tilde{p}}$-norm.
        
    Fix $i$. For every $p\in\mathcal{P}_i^\gamma$ (a prime in $(q_i^\gamma/2, q_i^\gamma]$), recall $E_{i,p}$ defined in \eqref{def-E-i-p} is now \begin{equation}\label{def-E-i-p-d=1}E_{i,p}:=\mathcal{N}_{q_i^{-1}}\left(\frac{\Z}{pq_i^\beta}\right)\cap[0,1].\end{equation}
    
    Since
    $$\supp\mu\subset\bigcup_{p\in\mathcal{P}_i^\gamma} E_{i,p},$$
    there exists $p_0\in\mathcal{P}_i^\gamma$ such that 
    \begin{equation}\label{pigeonholing-p}\mu(E_{i,p_0})\geq c\, q_i^{-\gamma}\log q_i=c\, q_i^{-b/2}\log q_i.\end{equation}

    Let $f:=\chi_{E_{i,p_0}}$. Then \begin{equation}\label{f-L-q-b<a}\|f\|_{L^q(\mu)}=\mu(E_{i,p_0})^{1/q}.\end{equation}
    For $\widehat{f\,d\mu}$, notice that for all $x\in E_{i,p_0}$ and all $\xi\in \mathcal{N}_{\frac{1}{100}}(p_0q_i^\beta\Z\cap[-q_i, q_i])$,
    $$\|x\xi\|:=\dist(x\xi,\Z)\leq \frac{1}{10}.$$
    This implies that, for all $\xi\in \mathcal{N}_{\frac{1}{10}}(p_0q_i^\beta\Z\cap[-q_i, q_i])$,
    $$|\widehat{f\,d\mu}(\xi)|=|\int e^{-2\pi i x\xi}f(x)\,d\mu(x)|\geq \int_{E_{i,p_0}}(\cos\frac{\pi}{5}-\sin\frac{\pi}{5})\,d\mu(x)\geq \frac{1}{10}\mu(E_{i,p_0}).$$
    Therefore
    $$\|\widehat{f\,d\mu}\|_{L^{\tilde{p}}}\geq 10^{-1}\mu(E_{i,p_0})\cdot |\mathcal{N}_{\frac{1}{10}}(p_0q_i^\beta\Z\cap[-q_i, q_i])|^{1/\tilde{p}}$$ $$\geq c_{\tilde{p}}\,\mu(E_{i,p_0})\cdot q_i^{(1-\gamma-\beta)/\tilde{p}}=c_{\tilde{p}}\,\mu(E_{i,p_0})\cdot q_i^{\frac{2-2a+b}{2\tilde{p}}}.$$
    Together with \eqref{f-L-q-b<a}, it follows that
    $$\frac{\|\widehat{f\,d\mu}\|_{L^{\tilde{p}}}}{\|f\|_{L^q}}\geq c_{\tilde{p}}\,\mu(E_{i,p})^{1/q'}\cdot q_i^{\frac{2-2a+b}{2\tilde{p}}}.$$
    
     Finally, by \eqref{pigeonholing-p}, when $q_i\rightarrow\infty$,
    $$\sup_{f\in L^q(\mu)}\frac{\|\widehat{f\,d\mu}\|_{L^{\tilde{p}}}}{\|f\|_{L^q}}\geq c_{\tilde{p},q}\, q_i^{-\frac{b}{2q'}+\frac{2-2a+b}{2\tilde{p}}}(\log q_i)^{1/q'}\rightarrow\infty,$$
    as desired.

\section{The non-geometric case and dimension of measures}\label{sec-dim-measure}
In this section, we prove Theorem \ref{thm-non-geo}. Recall $\dF\mu$ is defined in \eqref{def-F-dim-measure} and $\dH\mu$ is defined in \eqref{def-H-dim-measure}.

\subsection{$b\leq a$}
Although Theorem \ref{thm-non-geo} is came up for the non-geometric case $b>a$, it is valid for $b\leq a$ as well. Take $\mu$ to be the measure from Theorem \ref{thm-Fourier-dimension} with $2\gamma+\beta=a$ and $2\gamma=b$. Both $(i)$ and $(ii)$ have been verified in the previous section. It remains to check
\begin{itemize}
    \item $\dF\mu=2\gamma$;
    \item $\dH\mu=2\gamma+\beta$.
\end{itemize}

For the Fourier dimension, $\dF\mu=2\gamma$ because $\supp\mu$ is contained in a set of Fourier dimension $2\gamma$. 

For $\dH \mu$, first notice that
$$\underline{\dim}(\mu,x)\geq 2\gamma+\beta,\ \forall\,x\in\supp\mu,$$ 
due to its Frostman condition:
\begin{equation}\label{Frostman-implies-local-dim}\mu(B(x,r))\lesssim r^a\implies \liminf_{r\rightarrow 0}\frac{\log\mu(B(x,r))}{\log r}\geq a=2\gamma+\beta.\end{equation}

On the other hand,
$$\underline{\dim}(\mu,x)\leq 2\gamma+\beta, \,\mu\text{-a.e.}$$ because of the well known property (see, e.g. Proposition 10.3 in \cite{Falconer97})
$$\esssup_{x\sim\mu}\underline{\dim}(\mu,x)=\inf\{\dH E:\mu(E^c)=0\}\leq\dH\supp\mu.$$

Hence
$$\dH\mu=\inf_{x\in\supp \mu}\underline{\dim}(\mu,x)=2\gamma+\beta,$$
as desired.

\subsection{$a<b\leq 2a$}\label{subsec-b>a}
When $a<b\leq 2a$, one cannot directly take the measure $\mu$ from Theorem \ref{thm-Fourier-dimension}. But the more general measure constructed in Section \ref{sec-Fourier} with a carefully selected $\kappa\in[0,\gamma]$ works. More precisely, we take 
$$\gamma=b/2,\  \beta\in[0,a-\gamma] \text{ arbitrary}, \ \kappa=a-\gamma-\beta.$$ 
Notice $a-\gamma=a-b/2\geq 0$ and
$$0\leq \kappa<b-\gamma-\beta=\gamma-\beta\leq \gamma.$$
So the construction in Section \ref{sec-Fourier} is valid with these parameters.

Like the previous subsection, $\dF \mu=b$ follows because $\supp\mu$ is contained in a set of Fourier dimension $b$ by Theorem \ref{thm-Fourier-dimension}. Also the Frostman condition
$$\mu(B(x,r))\lesssim_\epsilon r^{a-\epsilon} $$
holds by \eqref{Frostman-kappa} with our choice $a=\kappa+\gamma+\beta$. Consequently Theorem \ref{thm-non-geo} (i) follows.

It remains to show $\dH\mu=a$ and Theorem \ref{thm-non-geo} (ii). We need the following lemma.

\begin{prop}\label{prop-lower-Frostman-non-geometric}
Let $\mu$ be the measure constructed in this subsection and suppose the auxiliary function $\phi$ in \eqref{def-phi-i-p} is positive on $[-\frac{1}{10}, \frac{1}{10}]$. Then for every $p_0\in \mathcal{P}_i^{\kappa,\gamma}$ and every $(10q_i)^{-1}$-interval $I$ contained in
\begin{equation}\label{smaller-nhd}\mathcal{N}_{(10q_i)^{-1}}\left(\frac{\Z\backslash p_0\Z}{p_0q_i^\beta}\right) \cap \bigcap_{k=1}^{i-1}\bigcup_{p\in\mathcal{P}_k^{\kappa,\gamma}}\mathcal{N}_{(10q_k)^{-1}}\left(\frac{\Z\backslash p\Z}{pq_k^\beta}\right),\end{equation}
we have
\begin{equation}\label{lower-Frostman}\mu(I) \gtrsim_\epsilon p_0^{-1}q_i^{-\gamma-\beta-\epsilon}.\end{equation} Furthermore, for all $p\in \mathcal{P}_i^{\kappa,\gamma}$,
    \begin{equation}\label{lower-AP}\mu(E_{i,p})\gtrsim_\epsilon q_i^{-\gamma-\epsilon},\end{equation}
    where $E_{i,p}$ is defined in \eqref{def-E-i-p-d=1}.
\end{prop}

We leave the proof of Proposition \ref{prop-lower-Frostman-non-geometric} to the end of this subsection. 

Now we can prove $\dH \mu=a$. By \eqref{Frostman-implies-local-dim}, the Frostman condition $\mu(B(x,r))\lesssim_\epsilon r^{a-\epsilon}$ implies
$$\underline{\dim}(\mu,x)\geq a,\ \forall\,x\in\supp\mu.$$
On the other hand, by choosing $p_0\approx q_i^{\kappa}$ in the lower bound \eqref{lower-Frostman} in Proposition \ref{prop-lower-Frostman-non-geometric}, one can see that, for all $\epsilon>0$, there exists $x\in\supp\mu$ with
$$\underline{\dim}(\mu,x)=\liminf_{r\rightarrow 0}\frac{\log\mu(B(x,r))}{\log r}\leq a+\epsilon.$$
Together we have $\dH \mu =a$, as desired.

It remains to show Theorem \ref{thm-non-geo} (ii), the sharpness of the Fourier restriction. Notice that the estimate \eqref{lower-AP} in Proposition \ref{prop-lower-Frostman-non-geometric} coincide with \eqref{pigeonholing-p}. Also the conditions on $\tilde{p}, q$ in Theorem \ref{thm-geo} and Theorem \ref{thm-non-geo} are the same. So the argument in Section \ref{sec-Sharpness} works for Theorem \ref{thm-non-geo} (ii) as well.

Finally we prove Proposition \ref{prop-lower-Frostman-non-geometric}.
\begin{proof}[Proof of Proposition \ref{prop-lower-Frostman-non-geometric}]

As $I$ is a $(10q_i)^{-1}$-interval,
our argument in Section \ref{subsec-Frostman} implies that it is equivalent to consider
$$\int_I G_i(y)\,dy,$$
where $G_i=\prod_{k=1}^i F_k$. By \eqref{def-phi-i-p}, \eqref{def-Phi-i-p}, for every $p\in \mathcal{P}_k^{\kappa,\gamma}$,
$$\Phi_{k,p}(y)\geq c_\phi\,p^{-1}q_k^{1-\beta}\cdot \chi_{\mathcal{N}_{(10q_k)^{-1}}\left(\frac{\Z\backslash p\Z}{pq_k^\beta}\right)}.$$
Notice that every $y$ in \eqref{smaller-nhd} associates a sequence of primes $p_1^y, p_2^y,\dots p_{i-1}^y$, with $p_k^y\in (q_k^{\kappa}, q_k^{\gamma}]$.
Therefore, by the separation condition \eqref{seperation-kappa} and the definition of $F_i$ in \eqref{def-F-i}, for every $p_0\in \mathcal{P}_i^{\kappa,\gamma}$ and $y$ in \eqref{smaller-nhd}, we have
   \begin{equation}\label{lower-G-i}|G_i(y)|\geq \frac{c_\phi\,p_0^{-1}q_i^{1-\beta}}{\#\mathcal{P}_i^{\kappa,\gamma}}\cdot \prod_{k=1}^{i-1}\frac{c_\phi\,(p_k^y)^{-1}q_k^{1-\beta}}{\#\mathcal{P}_k^{\kappa,\gamma}}\gtrsim_\epsilon p_0^{-1}q_i^{1-\beta-\gamma-\epsilon}.\end{equation}

Hence, for every $(10q_i)^{-1}$-interval $I$ contained in \eqref{smaller-nhd},
$$\mu(I)= \int_I G_i+error\gtrsim_\epsilon p_0^{-1}q_i^{-\beta-\gamma-\epsilon},$$
that proves \eqref{lower-Frostman}.

It remains to show \eqref{lower-AP}. We need to the number of $(10q_i)^{-1}$-intervals in \eqref{smaller-nhd}.

As $q_k$ is a rapidly increasing sequence, every $(10q_{k-1})^{-1}$-interval in $\R$ contains at least 
$$>c q_{k-1}^{-1}pq_k^\beta\geq c q_{k-1}^{-1}q_k^{\kappa+\beta} $$
many $(10q_k)^{-1}$-intervals from
$$\mathcal{N}_{(10q_k)^{-1}}\left(\frac{\Z\backslash p\Z}{pq_k^\beta}\right).$$
By the separation condition \eqref{seperation-kappa}, after taking the union of $p\in\mathcal{P}_{k}^{\kappa,\gamma}$, one can conclude that every $(10q_{k-1})^{-1}$-interval in $\R$ contains at least 
$$>c q_{k-1}^{-1}pq_k^\beta\#(\mathcal{P}_{k}^{\kappa,\gamma})\geq c q_{k-1}^{-1}q_k^{\kappa+\beta}\#(\mathcal{P}_{k}^{\kappa,\gamma})$$ many $(10q_k)^{-1}$-intervals from
$$\bigcup_{p\in\mathcal{P}_{k}^{\kappa,\gamma}}\mathcal{N}_{(10q_k)^{-1}}\left(\frac{\Z\backslash p\Z}{pq_k^\beta}\right).$$
Therefore the number of $(10q_i)^{-1}$-intervals in \eqref{smaller-nhd} is
\begin{equation}
    \label{lower-bound-total-interval}
    \geq c^i q_{i-1}^{-1}pq_i^{\beta}\prod_{k=1}^{i-1} q^{-1}_{k-1} q_k^{\kappa+\beta}\#\mathcal{P}_k^{\kappa,\gamma}
\end{equation}
Hence by \eqref{lower-Frostman}, \eqref{lower-bound-total-interval},
$$\mu(E_{i,p})\gtrsim_\epsilon q_i^{-\gamma-\epsilon}\cdot 
c^iq_{i-1}^{-1}\prod_{k=1}^{i-1} q^{-1}_{k-1} q_k^{\kappa+\beta}\#\mathcal{P}_k^{\kappa,\gamma}
\gtrsim_\epsilon q_i^{-\gamma-\epsilon},$$
as desired.
\end{proof}

\bibliographystyle{abbrv}
\bibliography{mybibtex}

\end{document}